%% file: main.tex
\documentclass[conference]{IEEEtran}
\usepackage{times}

\usepackage[numbers]{natbib}
\usepackage{multicol}
\usepackage[pagebackref=true,breaklinks=true,colorlinks,bookmarks=false]{hyperref}
\usepackage{amsthm}
\usepackage{adjustbox}
\usepackage{stackengine}
\usepackage{makecell}
\usepackage{caption}
\input{preamble_packages.tex}

\input{preamble_symbols.tex}

\input{shortcuts.tex}

\begin{document}

\title{On the Surprising Robustness of Sequential Convex Optimization for Contact-Implicit Motion Planning}

\author{\authorblockN{Yulin Li\authorrefmark{1}$^,$\authorrefmark{2},
Haoyu Han\authorrefmark{2},
Shucheng Kang\authorrefmark{2},
Jun Ma\authorrefmark{1},
and 
Heng Yang\authorrefmark{2}
}
\authorblockA{
\authorblockA{\authorrefmark{1}The Hong Kong University of Science and Technology}
\authorrefmark{2}Harvard University}
\vspace{1mm}
\authorblockA{\texttt{\url{https://computationalrobotics.seas.harvard.edu/CRISP}}}
}

\newcommand\blfootnote[1]{%
  \begingroup
  \renewcommand\thefootnote{}\footnote{#1}%
  \addtocounter{footnote}{-1}%
  \endgroup
}

\twocolumn[{%
\renewcommand\twocolumn[1][]{#1}%
\maketitle
\input{sections/fig-demos}
}]

\IEEEpeerreviewmaketitle

\input{sections/abstract.tex}
\input{sections/introduction.tex}
\input{sections/method.tex}

\input{sections/implementation-detail}
\input{sections/result.tex}
\input{sections/related-works}
\input{sections/conclusion.tex}


\bibliographystyle{plainnat}
\bibliography{refs}

\clearpage
\onecolumn
\appendices
\input{sections/app-proof}
\input{sections/app-formulation}
\input{sections/app-visualization}

\end{document}

%% file: preamble_packages.tex
\usepackage{comment}
\usepackage{siunitx}
\usepackage{relsize}
\usepackage{ifthen}
\usepackage[colorinlistoftodos]{todonotes}

\usepackage[vlined,ruled,linesnumbered]{algorithm2e}
\usepackage{graphics} 
\usepackage{rotating}
\usepackage{color}
\usepackage{enumerate}
\usepackage[T1]{fontenc}
\usepackage{psfrag}
\usepackage{epsfig} 
\usepackage{booktabs}
\usepackage{graphicx,url}
\usepackage{multirow}
\usepackage{array}
\usepackage{latexsym}
\usepackage{amsfonts}
\usepackage{amsmath}
\usepackage{amssymb}
\usepackage{mathtools}
\usepackage{xstring}
\usepackage[noend]{algorithmic}
\usepackage{multirow}
\usepackage{xcolor}
\usepackage{prettyref}
\usepackage{flexisym}
\usepackage{bigdelim}
\usepackage{breqn} 
\usepackage{listings}
\let\labelindent\relax
\usepackage{enumitem}
\usepackage{xspace}
\usepackage{bm}
\graphicspath{{./figures/}}
\usepackage{tikz}
\usetikzlibrary{matrix,calc}
\usepackage{tabularx}
\usepackage{subcaption}
\usepackage{float}

\usepackage{pgfplots}
\usepackage{pgfplotstable}
\usepgfplotslibrary{groupplots} 
\pgfplotsset{compat=1.18}       

\usepackage{mdwlist}

\let\itemize\undefined
\makecompactlist{itemize}{stditemize}


%% file: preamble_symbols.tex
\newrefformat{prob}{Problem\,\ref{#1}}
\newrefformat{def}{Definition\,\ref{#1}}
\newrefformat{sec}{Section\,\ref{#1}}
\newrefformat{sub}{Section\,\ref{#1}}
\newrefformat{prop}{Proposition\,\ref{#1}}
\newrefformat{app}{Appendix\,\ref{#1}}
\newrefformat{alg}{Algorithm\,\ref{#1}}
\newrefformat{cor}{Corollary\,\ref{#1}}
\newrefformat{thm}{Theorem\,\ref{#1}}
\newrefformat{lem}{Lemma\,\ref{#1}}
\newrefformat{fig}{Fig.\,\ref{#1}}
\newrefformat{tab}{Table\,\ref{#1}}
\newrefformat{assump}{Assumption\,\ref{#1}}

\newtheorem{theorem}{Theorem}

\newtheorem{assumption}[theorem]{Assumption}
\newtheorem{definition}[theorem]{Definition}
\newtheorem{proposition}[theorem]{Proposition}
\newtheorem{remark}[theorem]{Remark}
\newtheorem{example}[theorem]{Example}

\newcommand{\cf}{\emph{cf.}\xspace}

\newcommand{\bdmath}{\begin{dmath}}
\newcommand{\edmath}{\end{dmath}}
\newcommand{\beq}{\begin{equation}}
\newcommand{\eeq}{\end{equation}}
\newcommand{\bdm}{\begin{displaymath}}
\newcommand{\edm}{\end{displaymath}}
\newcommand{\bea}{\begin{eqnarray}}
\newcommand{\eea}{\end{eqnarray}}
\newcommand{\beal}{\beq \begin{array}{ll}}
\newcommand{\eeal}{\end{array} \eeq}
\newcommand{\beas}{\begin{eqnarray*}}
\newcommand{\eeas}{\end{eqnarray*}}
\newcommand{\ba}{\begin{array}}
\newcommand{\ea}{\end{array}}
\newcommand{\bit}{\begin{itemize}}
\newcommand{\eit}{\end{itemize}}
\newcommand{\ben}{\begin{enumerate}}
\newcommand{\een}{\end{enumerate}}

\newcommand{\calE}{{\cal E}}

\newcommand{\calI}{{\cal I}}

\newcommand{\hide}[1]{}

\newcommand{\hiddenText}{{\color{gray} hidden text.}}
\newcommand{\hideWithText}[1]{\hiddenText}

\newcommand{\subject}{\textup{ subject to }}

\DeclareMathOperator*{\argmin}{arg\,min}

\newcommand{\norm}[1]{\left\| #1 \right\|}

\newcommand{\tran}{^{\mathsf{T}}}

\newcommand{\Real}[1]{ { {\mathbb R}^{#1} } }

\newcommand{\scenario}[1]{{\smaller \sf#1}\xspace}

\newcommand{\blue}[1]{{\color{blue}#1}}

\newcommand{\linkToPdf}[1]{\href{#1}{\blue{(pdf)}}}
\newcommand{\linkToPpt}[1]{\href{#1}{\blue{(ppt)}}}
\newcommand{\linkToCode}[1]{\href{#1}{\blue{(code)}}}
\newcommand{\linkToWeb}[1]{\href{#1}{\blue{(web)}}}
\newcommand{\linkToVideo}[1]{\href{#1}{\blue{(video)}}}
\newcommand{\linkToMedia}[1]{\href{#1}{\blue{(media)}}}
\newcommand{\award}[1]{\xspace} 

\newcommand{\KwParameter}{\KwData}

%% file: shortcuts.tex
\renewcommand{\norm}[1]{\left\lVert #1 \right\rVert}

\newcommand{\bmat}{\left[ \begin{array}}
\newcommand{\emat}{\end{array}\right]}

\newcommand{\snopt}{\textsc{Snopt}\xspace}
\newcommand{\ipopt}{\textsc{Ipopt}\xspace}

\newcommand{\crisp}{\scenario{CRISP}}

%% file: sections/fig-demos.tex
\begin{minipage}{\textwidth}
\includegraphics[width=\linewidth,trim={0cm 0cm 0cm 0cm}, clip]{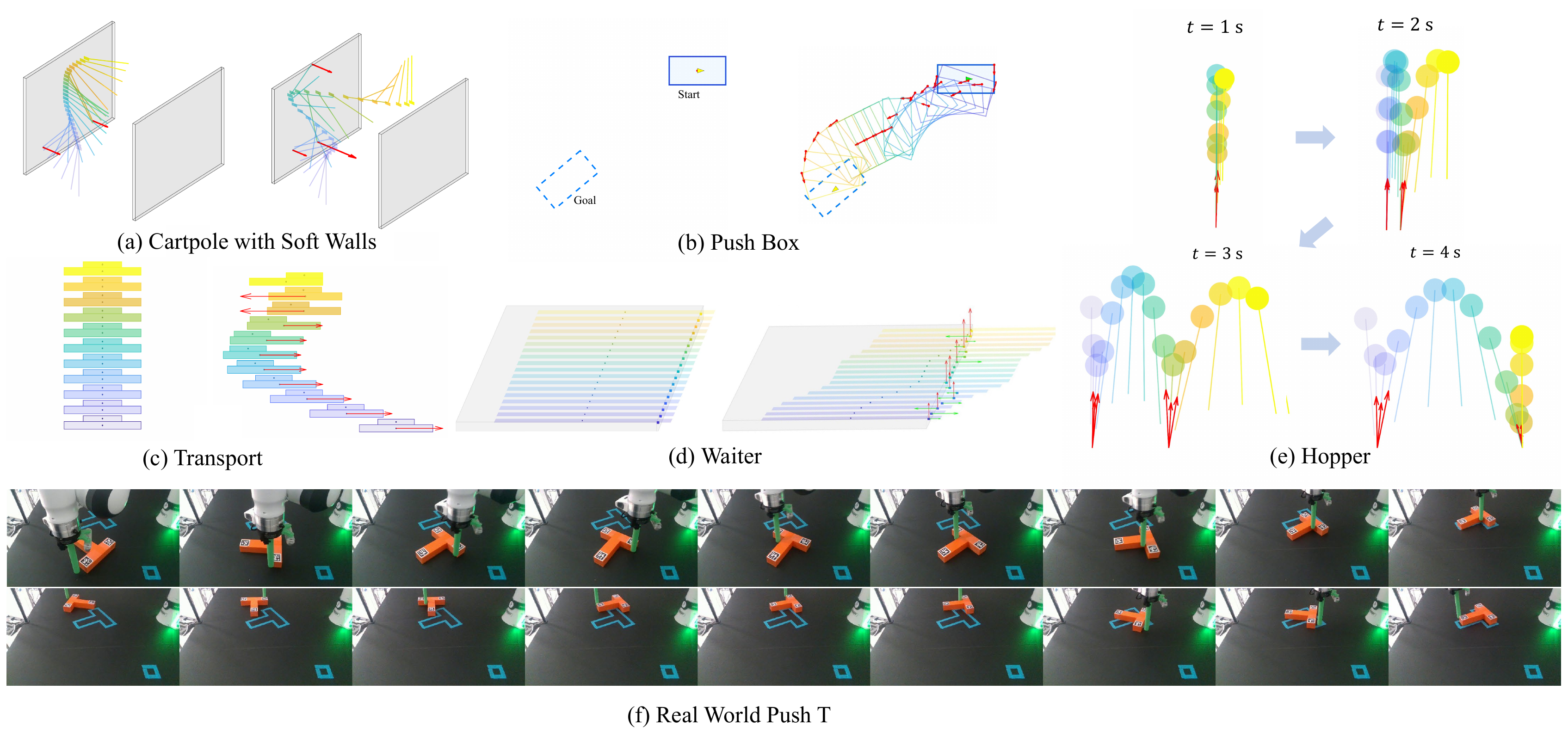}
\captionof{figure}{\crisp computes entirely new contact sequences from naive and even all-zero initializations. For (a), (b), (c), and (d), the left side shows the initial trajectories and the right side displays the optimized trajectory from \crisp. For (e) the hopper problem, the initial guess is a free-fall motion released from the origin. The color gradient represents the progression of time (from blue to yellow). For (f), we implement the policy derived from \crisp in a Model Predictive Control (MPC) framework for real-world push tasks. Detailed descriptions of these tasks are provided in \S\ref{sec:exp}.
\label{fig:demos}}
\vspace{1mm}
\end{minipage}

%% file: sections/abstract.tex

\begin{abstract}

Contact-implicit motion planning---embedding contact sequencing as implicit complementarity constraints---holds the promise of leveraging continuous optimization to discover new contact patterns online. Nevertheless, the resulting optimization, being an instance of Mathematical Programming with Complementary Constraints, fails the classical constraint qualifications that are crucial for the convergence of popular numerical solvers.  
We present \underline{r}obust \underline{c}ontact-\underline{i}mplicit motion planning with \underline{s}equential convex \underline{p}rogramming (\scenario{CRISP}), a solver that departs from the usual primal-dual algorithmic framework but instead focuses only on the primal problem. \crisp solves a convex quadratic program with an adaptive trust region radius at each iteration, and its convergence is evaluated by a merit function using weighted $\ell_1$ penalty. We (\emph{i}) prove sufficient conditions on \crisp's convergence to first-order stationary points of the merit function; (\emph{ii}) release a high-performance C++ implementation of \crisp with a generic nonlinear programming interface; and (\emph{iii}) demonstrate \crisp's surprising robustness in solving contact-implicit planning with naive initializations. In fact, \crisp solves several contact-implicit problems with an all-zero initialization.

\end{abstract}


%% file: sections/introduction.tex
\section{Introduction}
\label{sec:introduction}

Robots must intelligently establish and disengage \emph{contacts} to successfully perform complex tasks in the physical world, such as manipulation of daily objects and locomotion in rough terrains.
However, the hybrid nature of combining continuous dynamics with discrete contact events, the discontinuities in force profiles, and the potential for stick-slip transitions, make it notoriously difficult to ``plan through contact''.

\textbf{Contact-implicit motion planning.} Among the many efforts for motion planning through contact (see \S\ref{sec:relatedworks} for a review), the so-called \emph{contact-implicit} formulation stood out as a particularly popular and promising approach~\cite{posa2014ijrr-traopt-directmethod-contact,kerim2006TO}. This formulation distinguishes itself by integrating contact dynamics into the motion planning framework through the introduction of \emph{complementarity constraints}, which allow simultaneous optimization of the state-control trajectories and the contact forces without explicit (and discrete) mode switching (\cf~\eqref{eq:contact-implicit}). 
However, the contact-implicit formulation circumvents explicit mode switching at the price of arriving at an optimization problem known as \emph{Mathematical Programming with Complementarity Constraints} (MPCC). Although they ``look like'' smooth and continuous optimization problems, MPCC problems are known to fail the classical constraint qualifications (CQs) at \emph{all feasible points} (e.g., Linear Independence Constraint Qualification and Mangasarian-Fromovitz Constraint Qualification)~\cite{MOR-scheel2000mathematical,SIOPT-Scholtes-2001,SIOPT-Ye-MPCC-2016}. 
To make this concrete, we provide the geometric intuition through a simple MPCC problem.
\begin{example}[Geometric Intuition of MPCC]\label{example:toy}
Consider the following two-dimensional MPCC problem:
\begin{subequations}\label{eq:toy}
    \begin{align}
\min_{(x_1,x_2) \in \mathbb{R}^2} &\; f(x) := x_1^2 + x_2^2 \\
\subject &\; g_1(x) := x_1 \geq 0, \\
&\; g_2(x) := x_2 \geq 0, \\
&\; g_3(x) := x_1 \cdot x_2 = 0. \label{eq:toy:cc}
\end{align}
\end{subequations}
where~\eqref{eq:toy:cc} is a complementarity constraint.
As depicted in~\prettyref{fig:toy_example},
the objective level sets are circles in $\mathbb{R}^2$, and
the feasible set consists of the union of the two coordinate axes in the first quadrant
\bea
\{x\in\Real{2} \mid x_1 \geq 0, x_2 =0\} \cup \{x \in \Real{2} \mid x_1 = 0, x_2 \geq 0\}.
\eea
The optimal solution of problem~\prettyref{eq:toy} lies at the origin $(0,0)$.
\begin{figure}[t]
    \centering
    \includegraphics[width=0.8\linewidth, trim={0cm 12cm 25cm 0cm}, clip]{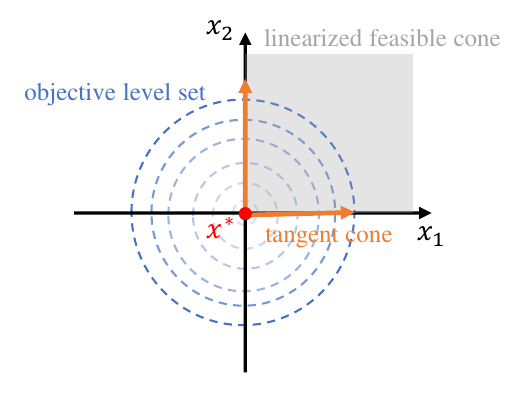}
    \vspace{-6mm}
    \caption{
    Geometric intuition of MPCC through Example~\prettyref{example:toy}.}
    \label{fig:toy_example}
    \vspace{-4mm}
\end{figure}

To understand the local geometry at the origin, we need to characterize two types of ``cones'' around the origin.

(a) The \emph{tangent cone (TC)} at a point is defined as the set of tangent vectors of all curves approaching the point from the feasible region~\cite{nocedal1999springer-numerical-optimization}. 
Thus, the tangent cone at \((0, 0)\) is:
\bea
\begin{split}
    TC=\{d \in \mathbb{R}^2 \mid d_1 \geq 0, d_2 = 0\} \ \cup\  \\ \{d \in \Real{2} \mid d_1 = 0, d_2 \geq 0\}.
\end{split}
\eea
This consists of two directions: one along the positive \(x_1\)-axis and one along the positive \(x_2\)-axis, see Fig.~\ref{fig:toy_example}.

(b) The \emph{linearized feasible cone (LFC)} is formed by all active constraints at that point~\cite{nocedal1999springer-numerical-optimization}:
\bea
LFC = \left\{d \in \mathbb{R}^2 \;\bigg|\; 
\begin{array}{l}
\nabla g_i^\top d \geq 0,i=1,2 \\
\nabla g_j^\top d = 0 , j=3
\end{array}
\right\}.
\eea
Thus, the linearized feasible cone at \((0, 0)\) is:
\bea
LFC = \{d \in \Real{2} \mid d_1 \geq 0, d_2 \geq 0\},
\eea
i.e., the entire first quadrant. Therefore, we conclude that $TC \subsetneq LFC$---the TC is a strict subset of the LFC at $(0,0)$.

We remark that the TC is ``geometric'', while the LFC is ``algebraic''. While the definition of LFC involves the algebraic constraints and their linearizations, the definition of TC characterizes the local geometry regardless of how the feasible set is algebraically parameterized. 
\end{example}

After noticing the gap between the TC and the LFC, the reader might wonder why this is a problem. The reason lies in that, the equivalence of TC and LFC---typically ensured by CQs---is a prerequisite to establish the Karush–Kuhn–Tucker (KKT) conditions for local optimality \cite{nocedal1999springer-numerical-optimization}. Moreover, almost all primal-dual nonlinear programming (NLP) solvers (e.g., \snopt~\cite{SNOPT}, \ipopt~\cite{ipopt}) rely on the KKT optimality conditions---they search for a pair of primal and dual variables satisfying the KKT conditions. Therefore, the inherent absence of CQs and the gap between TC and LFC of the MPCC problems break the theoretical foundations for primal-dual algorithms. The numerical consequence of this, as we will show in \S\ref{sec:exp}, is that primal-dual solvers can exhibit poor convergence, stuck in infeasible solutions, and/or numerical instabilities when solving contact-implicit MPCC problems. 


\textbf{Relaxation?} An extensively studied strategy to mitigate the failure of CQs is to ``relax'' the complementarity constraint. In~\cite{MP-FACCHINEI-1999}, using Example~\ref{example:toy} as illustration, the authors relaxed the complementarity constraint~\eqref{eq:toy:cc} as $x_1\cdot x_2 = \alpha$ for $\alpha >0$. In~\cite{SIOPT-Scholtes-2001,MOR-STEFAN-2001}, a different relaxation scheme where $x_1 \cdot x_2 \leq \alpha$ for $\alpha > 0$ is proposed. In both relaxations, CQs are restored (i.e., no gap between the TC and the LFC) and KKT optimality conditions are reassured. However, since $\alpha=0$ is the original non-relaxed problem we want to solve, a homotopy scheme is needed to start with $\alpha$ very large and gradually push $\alpha$ to zero. There are two issues with the relaxation approach. First, one needs to solve a sequence of nonconvex problems instead of just one, making this approach computationally expensive. Second, while the problem with $\alpha$ very large is generally well conditioned, as $\alpha \downarrow 0$, numerical issues appear again. For these reasons, to the best of our knowledge, there does not exist a well-accepted implementation of the relaxation approach. Recently in robotics, \cite{le2024fast} targeted at linear complementarity constraints---which are indeed KKT optimality conditions of a lower-level convex quadratic program (QP)---and designed a bilevel optimization algorithm where the lower-level QP is solved in a differentiable way to provide gradient for the upper-level trajectory optimization problem. Crucially, to ensure the QP is differentiable with respect to contact, relaxation (or smoothing) is applied again, but this time to the interior point algorithm used for solving the QP. However, (a) it is unclear whether the approach extends to nonlinear complementarity constraints studied in our paper; and (b) the implementation provided by~\cite{le2024fast} does not follow a generic nonlinear programming interface. Therefore, we did not benchmark against~\cite{le2024fast} in \S\ref{sec:exp}.

\emph{Can we solve contact-implicit motion planning in a numerically robust way, without relaxation?}
\begin{figure*}[tp]
    \centering
    \includegraphics[width=1\linewidth, trim={0cm 0cm 0cm 0cm}, clip]{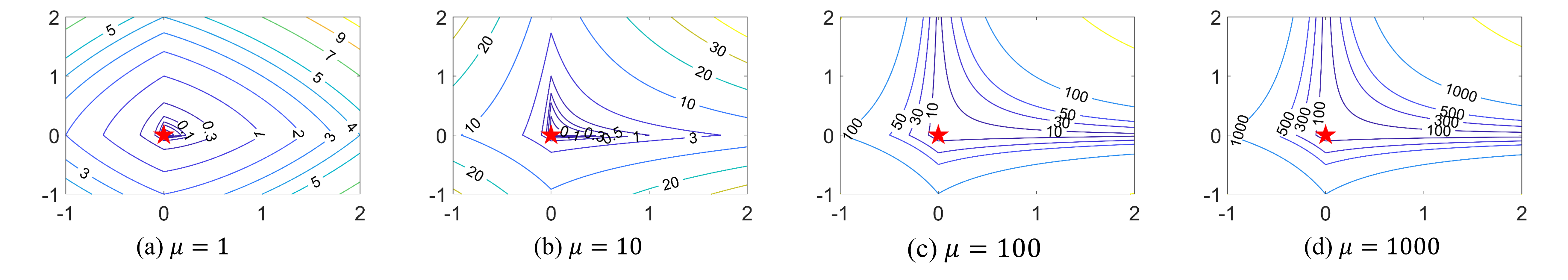}
    \vspace{-4mm}
    \caption{
    Depiction of the $\ell_1$ penalty merit function for Example~\prettyref{example:toy} with different penalty parameters $\mu$.}
    \label{fig:toy_example_merit_function}
    \vspace{-4mm}
\end{figure*}

\textbf{Contributions.}
The answer is rather discouraging if considering a generic nonconvex MPCC problem, as highlighted by the mathematical optimization literature~\cite{fletcher2000practical, SIOPT-FLETCHER-2006, OMS-Fletcher-2004}. Nevertheless, contact-implicit planning problems in robotics possess a unique property: the objective function is often \emph{convex} (and quadratic), as the goal typically involves tracking a trajectory or reaching a target. Formally, we consider the problem:
\begin{subequations}\label{eq:contact-implicit}
    \begin{align}
    \min_{v, \lambda} \quad & J(v, \lambda) \\
    \subject \quad & f(v, \lambda) = 0, \label{eq:ci:dynamics}\\
    & c_i(v, \lambda) \geq 0, \ i \in \calI  \label{eq:ci:ineq}\\
    & c_i(v, \lambda) = 0, \ i \in \calE \label{eq:ci:eq} \\
    & 0 \leq \phi(v,\lambda) \perp \lambda \geq 0,\label{eq:ci:complementarity}
\end{align}
\end{subequations}
where $v$ includes the trajectory of both robot state and control inputs, $\lambda$ denotes the complementarity variables typically associated with contact forces. We assume the objective function $J(v, \lambda)$ is convex---usually a quadratic loss function of tracking errors and control efforts. Constraint \eqref{eq:ci:dynamics} represents the nonlinear system dynamics. Constraints \eqref{eq:ci:ineq} and \eqref{eq:ci:eq} are general inequality and equality constraints. The nonlinear complementarity constraint \eqref{eq:ci:complementarity} enforces the relationship between $\lambda$ and the nonlinear function $\phi(v,\lambda)$. The expression ``$\phi(x,\lambda) \perp \lambda$'' should be interpreted as $\phi_i(x,\lambda) \cdot \lambda_i = 0$, or its equivalent form $\phi_i(x,\lambda) \cdot \lambda_i \leq 0$ as suggested by~\cite{SIOPT-FLETCHER-2006}, for every entry of $\phi(\cdot)$ and $\lambda$. It should be emphasized that problem~\eqref{eq:contact-implicit} is nonlinear and nonconvex, as we impose no restrictions on the convexity of the constraints. This formulation is classified as a nonlinear complementarity problem (NCP) due to constraints \eqref{eq:ci:complementarity}. A key distinction between NCP and linear complementarity problem (LCP) lies in the nonlinearity present in both the complementarity constraints and other constraints. While one can use $s = \phi(v,\lambda)$ to recast complementarity constraints as $\lambda \perp s$, this transformation merely shifts the nonlinearity to the additional constraint $s = \phi(v,\lambda)$. In the linear complementarity problem (LCP) literature, both $\phi$ and other constraints are typically required to be linear (as in~\cite{aydinoglu2023icra-realtime-multicontact-mpc-admm} and LCQpow~\cite{Hall_2024}). In contrast, \crisp is capable of handling generic nonlinear dynamics, as demonstrated in our examples.
We will give concrete examples of~\eqref{eq:contact-implicit} in \S\ref{sec:exp}. For now, let us work with the generic template~\eqref{eq:contact-implicit}.

Our major contribution is to depart from the usual primal-dual algorithmic framework---due to the failure of CQs in MPCC---and propose a \emph{primal-only} algorithm named \crisp (\textbf{\underline{R}}obust \textbf{\underline{C}}ontact-\textbf{\underline{I}}mplicit motion planning with \textbf{\underline{S}}equential convex \textbf{\underline{P}}rogramming). \crisp features a merit function with weighted $\ell_1$ penalty to measure primal feasibility and objective reduction. It solves a series of trust-region convex quadratic programs (QPs) formulated using local second-order information of the objective and \emph{first-order} information of the constraints. In stark contrast with the well-known sequential quadratic programming (SQP) framework, \crisp (a) does not require second-order information of the constraints, and (b) does not maintain and update dual variables. Since the objective function is convex, the inner-loop QP subproblem inherits convexity by construction. These make \crisp both simple to implement and computationally efficient. 

One then wonders whether a simple algorithm like \crisp can offer any convergence guarantees. We show that the convexity of the objective function allows us to prove sufficient conditions under which \crisp can guarantee convergence to the stationary points of the merit function. Importantly, the sufficient conditions are numerically verifiable, and hence the convergence can be ``certified'' from the numerical iterations of \crisp. Moreover, we show that the local minima of the merit function are indeed the local minimizers of the original problem if feasible for the original problem to close the loop.

Finally, we open-source a high-performance C++ implementation of \crisp. Our C++ implementation follows a generic nonlinear programming interface where the user defines objective function and constraints. We leverage automatic differentiation to obtain gradient and Hessian information. We then apply \crisp to solve six contact-implicit motion planning problems and benchmark its performance against both generic and robotics-specific solvers. We believe the way we model some of the contact-implicit planning problems is also new. Our numerical results demonstrate that \crisp consistently generates non-trivial and entirely new contact sequences from naive and even all-zero initializations, while existing solvers can struggle to find even feasible solutions. 

To summarize, our contributions are:
\begin{itemize}
    \item \textbf{Theory and Algorithm}. We propose a primal-only algorithm for contact-implicit motion planning called \crisp and provide theoretical guarantees.
    \item \textbf{Implementation}. We release an open-source high-performance C++ implementation of \crisp.
    \item \textbf{Benchmark}. We model six contact-rich planning problems using the contact-implicit formulation and benchmark \crisp against existing solvers.
\end{itemize}



\textbf{Paper organization.} We present the theory and algorithm for \crisp in \S\ref{sec:method}, its implementation details in \S\ref{sec:implementation}, and benchmark results in \S\ref{sec:exp}. We postpone the related work review to \S\ref{sec:relatedworks} and conclude in \S\ref{sec:conclusion}. Appendix and project website provide proofs and the detailed contact-implicit formulation for the five planning problems, as well as the animations of numerically computed motion trajectories. 

%% file: sections/method.tex

\section{\crisp: Theory and Algorithm}
\label{sec:method}
Clearly, the contact-implicit planning problem~\eqref{eq:contact-implicit} can be transformed into a generic NLP problem, which also aligns with our implemented C++ solver. Let $x = [v, \lambda]$, we denote the NLP formulation of~\eqref{eq:contact-implicit} as:
\begin{subequations}\label{eq:problem_general} 
    \begin{align}
        \min_x &\quad J(x) \\
        \subject &\quad c_i(x) = 0, \quad i \in \mathcal{E} \label{eq:general:eq}\\
        &\quad c_i(x) \ge 0. \quad i \in \mathcal{I} \label{eq:general:ineq}
    \end{align}
\end{subequations}
Here, we absorb all other constraints into the general template $c$ for simplicity.
\textbf{Primal-only merit function.}
From a primal-only perspective, the core optimization challenge is to balance descent of the objective function with satisfaction of the constraints~\cite{nocedal1999springer-numerical-optimization}. 
In \crisp, we adopt the following merit function with weighted $\ell_1$ penalty to evaluate the quality of a point $x$:
\begin{align}
    \label{eq:merit-func}
    \phi_1(x;\mu) \triangleq J(x) +  \sum_{i \in \mathcal{E}} \mu_i| c_i(x) | + \sum_{i \in \mathcal{I}} \mu_i [c_i(x)]^-,
\end{align}
where $\mu_i > 0$ is the penalty parameter for each constraint $i$. The notation $[c_i]^- \triangleq \max\{0, -c_i\}$ is used to properly penalize inequality constraint violations. 

\begin{example}[Merit Function for the Toy Problem]\label{ex:toy:merit}
    Continuing Example~\ref{example:toy}, we plot the merit function for the toy MPCC problem~\eqref{eq:toy} with different penalty parameters in Fig. \ref{fig:toy_example_merit_function}. Evidently, the stationary points of the merit function align with the global minimum of Example~\ref{example:toy} at the origin.
\end{example}


\begin{remark}[Merit Function]
    The $\ell_1$ penalty function is a class of nonsmooth penalty functions that has demonstrated remarkable success in practical applications. It is exact in the sense that, given an appropriate penalty parameter and under mild constraint qualifications, local optimizers of the original problem are also minimizers of the merit function~\cite{Fletcher1983}.

    We need to emphasize that, due to the lack of CQs in the MPCC problem, it is not fully clear whether minimization of the merit function~\eqref{eq:merit-func} has the same type of guarantees as in the case where CQs hold. However, the intuition that the merit function balances objective reduction and constraint satisfaction still holds, and according to Example~\ref{ex:toy:merit}, the minimization of the merit function leads to the optimal solution, at least for the toy problem. We give a sufficient condition of local optimality in Proposition \ref{prop:local_minima} and leave a precise investigation of the relationship between the merit function and the original problem in the case of MPCC for future work.
    

    A difference of our merit function from the usual $\ell_1$ penalty is that we maintain a separate penalty parameter $\mu_i$ for each constraint. This individualized penalization strategy often leads to superior convergence in practice.

\end{remark}

\textbf{Convex subproblem.} 
To minimize the merit function~\prettyref{eq:merit-func}, we adopt a sequential optimization strategy that solves a series of subproblems by approximating a local model of the merit function in the vicinity of the current iteration. Formally, our subproblem seeks to minimize the following quadratic model:
\begin{align}\label{eq:subproblem-nonsmooth}
\min_{p_k} \, q_{\mu,k}(p_k) := & \quad J_k + \nabla J_k^\top p_k 
+ \frac{1}{2} p_k^\top \nabla^2_{xx} J_k \, p_k 
\notag\\
&\quad + \sum_{i \in \mathcal{E}}\mu_i \lvert c_i(x_k) + \nabla c_i(x_k)^\top p_k \rvert \notag\\ &\quad + \sum_{i \in \calI}\mu_i  \big[c_i(x_k) + \nabla c_i(x_k)^\top p_k\big]^-,
\end{align}
where $J_k$, $\nabla J_k$, and $\nabla_{xx}^2 J_k$ represent the objective function's value, gradient, and Hessian at the current iterate $x_k$, respectively. The constraints $c_i$ are linearized around $x_k$. In the subproblem~\eqref{eq:subproblem-nonsmooth}, ``$p_k$'' is known as the \emph{trial step}.

From~\eqref{eq:subproblem-nonsmooth}, the readers see that our method diverges from the subproblem formulation in the classical SQP methods~\cite{sqp-method, fletcher2000practical, jordana2023stagewise}, in the sense that~\eqref{eq:subproblem-nonsmooth} refrains from computing second-order derivatives of the constraints and the Lagrangian. Instead, we linearize the constraints while maintaining a quadratic model of the objective function only. 
Our approach aligns with the concept of sequential convex programming discussed in \cite{scp-01,scp-02,scp-03,mao2018successive,mao2016successiveconvexification}, including the well-known robotics package \textsc{TrajOpt} \cite{trajopt}. \textsc{TrajOpt} implements sequential convex programming for trajectory optimization and convex collision detection. However, it lacks user-friendly APIs to include contact constraints and does not support weighted $\ell_1$ merit functions or second-order corrections, nor does it provide convergence analysis. Beyond the SCP foundation, we aim for a high-performance implementation featuring carefully engineered modules and a generalized interface designed for broad contact-rich motion planning applications, while providing convergence guarantees in the contact-implicit context where CQs assumptions do not hold.

Leveraging the convexity of the objective function $J$ and introducing the penalty term, the subproblem~\eqref{eq:subproblem-nonsmooth} is always convex and feasible. This offers several advantages. First, it reduces the computational burden associated with calculating second-order derivatives of constraints, which is particularly beneficial for problems with complex constraints (e.g., in robotics). Second, the convexity and guaranteed feasibility of the subproblem not only facilitate efficient solution of~\eqref{eq:subproblem-nonsmooth} using well-established convex optimization techniques but also enhance the robustness of the overall algorithm.

\begin{remark}[Elastic Mode]
    In conventional sequential optimization methods, such as SQP, constraints are typically linearized directly from the original problem. However, when applied to MPCC, this approach inevitably leads to infeasible subproblems~\cite{SIOPT-FLETCHER-2006,SIOPT-ANITESCU-2005}. To address this issue, some numerical solvers, like \snopt, implement an ``Elastic Mode'' where they add penalty variables to move constraints into the objective function upon detecting subproblem infeasibility.
    Our approach aligns with the motivation of the elastic mode in the sense that the merit function~\eqref{eq:merit-func} is ``elastic'' from the very beginning because the constraints are penalized in the objective function.
    This strategy ensures our subproblems are always feasible by design.
\end{remark}



\textbf{Globalization.}
The only issue with the subproblem~\eqref{eq:subproblem-nonsmooth} is that it may be unbounded from below (i.e., the optimal value is $-\infty$).
To avoid this issue, we introduce the trust-region approach following \cite{nocedal1999springer-numerical-optimization}. Specifically, we constrain the $\ell_{\infty}$ norm of the trial step $p_k$. Formally, the trust-region subproblem with $\ell_\infty$ norm constraint becomes:
\begin{subequations} \label{eq:subproblem-smooth}
    \begin{align}
        \min_{p_k, v, w, t} & \displaystyle \,\,\,J_k + \nabla J_k\tran p_k  + \frac{1}{2} p_k\tran \nabla_{xx}^2 J_k p_k \nonumber \\
        & + \mu \sum_{i \in \mathcal{E}} (v_i + w_i) + \mu \sum_{i \in \mathcal{I}} t_i  \\
        \text{subject to} &\,\, \nabla c_i(x_k)\tran p_k + c_i(x_k) = v_i - w_i, \ i \in \mathcal{E} \label{eq:subproblem-smooth-eq}\\
        & \,\, \nabla c_i(x_k)\tran p_k + c_i(x_k) \ge -t_i, \ i \in \mathcal{I} \label{eq:subproblem-smooth-ineq}\\
        & \,\, v, w, t \ge 0 \\
        & \,\, \|p_k\|_\infty \le \Delta_k, \label{eq:subproblem-trust-region}
    \end{align}
\end{subequations}
where~\eqref{eq:subproblem-trust-region} represents the $\ell_\infty$ norm constraint with trust-region radius $\Delta_k$. From~\eqref{eq:subproblem-nonsmooth} to~\eqref{eq:subproblem-smooth}, we have also introduced slack variables $v,w,t$ to reformulate the nonsmooth $\ell_1$ objective as smooth constraints. Evidently, the subproblem~\eqref{eq:subproblem-smooth} is a standard convex QP.

Essentially, the trust region constraint~\eqref{eq:subproblem-trust-region} acts as a filter on how much we trust the local model (quadratic objective and linearized constraints) to approximate the merit function. The $\ell_{\infty}$ norm only limits the maximum step length in the trial step, avoiding the issues related to ill-conditioning that can occur in line search methods.

To determine whether a trial step should be accepted and whether the trust-region radius should be adjusted, we monitor the ratio of actual decrease to predicted decrease:
\begin{align}
    \label{eq:reduction_ratio}
    \rho_k \triangleq \frac{\text{ared}_k}{\text{pred}_k} 
    = \frac{
        \phi_1(x_k; \mu) - \phi_1(x_k + p_k; \mu)
    }{
        q_{\mu,k}(0) - q_{\mu,k}(p_k) 
    }.
\end{align}
This ratio $\rho_k$ serves as a key indicator of the quality of the trial step. As $\rho_k$ approaches 1, it signifies that the trial step is increasingly satisfactory, indicating the local quadratic model at this step accurately describes the local behavior of the true merit function within the trust region radius. When the actual reduction is negative or the reduction ratio is very low, it indicates that the local model is highly inaccurate. In such cases, we typically shrink the trust region radius. To avoid continuous shrinking of the trust region and the Maratos effect, we introduce a second order correction step. The essence of this correction is to replace the linear approximation of $c(x_k+p)$ in subproblem~\prettyref{eq:subproblem-smooth} with a second-order model:
\begin{equation}
c(x_k+p) \approx c(x_k) + \nabla c(x_k)^T p + \frac{1}{2}p^T \nabla_{xx}^2c(x_k) p.
\end{equation}

Suppose we have calculated an unsatisfactory trial step $p_k$, and the correction step to be determined is not far from $p_k$. 
We approximate the second-order term $p^T \nabla_{xx}^2c(x_k) p$ utilizing the value $c(x_k + p_k)$ computed at the trial step:
\begin{align}\label{eq:second_order_correction}
    p^T \nabla_{xx}^2c(x_k) p \approx \,& p_k^T \nabla_{xx}^2c(x_k) p_k \notag\\ 
    \approx \,& c(x_k+p_k) - c(x_k) - \nabla c(x_k)^\top p_k.
\end{align}
From~\prettyref{eq:second_order_correction}, we use $c(x_k+p) = \nabla c(x_k)^{\top} p + (c(x_k+p_k) -  \nabla c(x_k)^\top p_k)$ to replace the first-order approximation in~\eqref{eq:subproblem-smooth-eq} and~\eqref{eq:subproblem-smooth-ineq} to calculate a correction step $\Bar{p}_k$.
This approach does not require explicit second-order information of the constraints, thereby maintaining computational efficiency. Furthermore, it preserves the convex QP structure of the subproblem. In practice, this second-order correction has proven highly effective in correcting inaccuracies of the linearized model, leading to more efficient reductions in the merit function.

\textbf{Convergence analysis.} 
We now provide convergence guarantees of \crisp to stationary points of the merit function~\eqref{eq:merit-func}. We first recall the definition of the directional derivative.
\begin{definition}[Directional Derivative]
    The directional derivative of a function $f$: $\mathbb{R}^n \rightarrow \mathbb{R}$ in the direction $p$ is:
        \begin{equation}
        D(f;p) = \lim_{t\rightarrow0} \frac{f(x + tp)-f(x)}{t}.
    \end{equation}
\end{definition}

The stationary point of the nonsmooth merit function is:
\begin{definition}[Stationary Point]
A point $\Bar{x}$ is a stationary point of the merit function $\phi_1(x,\mu)$ in~\eqref{eq:merit-func} if its directional derivatives are nonnegative along all directions $p$, i.e.,
    \begin{equation}
        D(\phi_1(\Bar{x};\mu);p)\geq 0.
    \end{equation}
\end{definition}

We then make mild assumptions about the motion planning problems under consideration.
\begin{assumption}[Convexity]\label{assump:convergence_assump}
In problem~\prettyref{eq:problem_general}, the objective function $J(x)$ is convex and continuous differentiable; $c_i$ is differentiable and $\nabla c_i$ is Lipschitz continuous for all $i\in\calE$ and $i\in\calI$.
\end{assumption}
We are ready to present the main result.
\begin{theorem}[Convergence]
\label{thm:local_convergence}
Under Assumption~\ref{assump:convergence_assump}, suppose:
\begin{enumerate}
    \item The sequence of convex subproblems~\eqref{eq:subproblem-smooth} converges, i.e., $x_k\rightarrow x^\star$ for some point $x^\star$
    \item The trust region radius remains bounded above zero, i.e., there exists $\Delta_{\min} > 0$ such that $\Delta_k \geq \Delta_{\min}$ for all $k$.
\end{enumerate}
Then, $x^\star$ is a stationary point of the merit function~\eqref{eq:merit-func}.
\end{theorem}
\begin{proof}
    See~\prettyref{app:proof_main_thm}.
\end{proof}

While the nonsmoothness of the merit function makes it challenging to verify whether its stationary points are local minima, a correspondence between the local minima of the merit function and the original problem can be established.

\begin{proposition}[Local Optimality]
\label{prop:local_minima}
Let $x^\star$ be a local minimizer of the merit function $\phi_1(x;\mu)$ in \eqref{eq:merit-func}. If $x^\star$ is feasible for the original problem \eqref{eq:problem_general}, that is,
\begin{align}
    \quad c_i(x^\star) = 0, \quad i \in \mathcal{E}\\
    \quad c_i(x^\star) \geq 0. \quad i \in \mathcal{I}
\end{align}
Then $x^\star$ is also a local minimizer of the original problem \eqref{eq:problem_general}.
\end{proposition}
\vspace{-2mm}
\begin{proof}
    Since $x^\star$ is a local minimizer of $\phi_1(x;\mu)$, by definition, there exist $r>0~~s.t.~~\phi_1(x^\star;\mu) \leq \phi_1(x;\mu),~~\forall||x-x^\star||<r$. Define the set $||x-x^\star||<r$ as $B_r(x^\star)$ and let the feasible set of the original problem \eqref{eq:problem_general} be $\Omega$. Then, we have
    $$\phi_1(x^\star;\mu) \leq \phi_1(x;\mu).~~\forall x \in B_r(x^\star) \cap \Omega$$
    On the other hand, we have $\phi_1(x;\mu) = J(x)$ $\forall x\in\Omega$. Thus, evaluating the objective function of the original problem, we have:
    $$J(x^\star) \leq J(x),~~\forall x \in B_r(x^\star) \cap \Omega, ~~x^\star \in \Omega$$
    which is the definition of a local minimizer. 
\end{proof}

Combining Theorem \ref{thm:local_convergence} and Proposition \ref{prop:local_minima}, we conclude that if (a) \crisp converges to a stationary point $x^\star$ of the merit function, (b) $x^\star$ is a local minimizer of the merit function, and (c) $x^\star$ is feasible for the original problem \eqref{eq:problem_general}, then $x^\star$ is also a local minimizer of \eqref{eq:problem_general}.
In other words, when \crisp computes a stationary point $x^\star$ of the merit function (which often occurs and can be certified) that is also locally optimal, the only way that $x^\star$ might fail to be a local minimizer of the original problem \eqref{eq:problem_general} is that $x^\star$ is infeasible for \eqref{eq:problem_general}. We note that while this statement is theoretically useful, due to the nonsmoothness of the merit function, certifying local optimality of $x^\star$ for the merit function is challenging. We leave this for future work.

\textbf{Final algorithm.}
Having introduced the key algorithmic components and the convergence analysis, we formally present \crisp in Algorithm~\ref{alg:crisp}. 
The algorithm initiates with a common penalty variable $\mu_0$ for all constraints. Guided by the checkable convergence conditions outlined in~\prettyref{thm:local_convergence}, we solve a series of trust-region subproblems (convex QPs) in the inner iterations. This process continues until the conditions in Theorem~\ref{thm:local_convergence} are satisfied, indicating convergence to a stationary point of the current merit function. 
In the outer iterations, we examine the constraint violation against a threshold $\epsilon_c$. For constraint violations exceeding this threshold, we increase their corresponding penalty parameters.
The default values of all hyperparameters in Algorithm~\ref{alg:crisp} are summarized in~\prettyref{tab:hyperparameters}.

%% file: sections/implementation-detail.tex
\section{\crisp: Implementation}
\label{sec:implementation}

\textbf{User interface.}
\crisp follows the general nonlinear programming problem formulation in~\eqref{eq:problem_general} where the user defines the objective function and constraints. Our implementation builds upon the fundamental data types in Eigen3 (vectors and matrices), wrapped with CPPAD~\cite{cppad} and CPPAD Code Generation (CG)~\cite{cppadcg}. This architecture allows users to define problems in an Eigen-style syntax, specifying objectives and constraints (both equality and inequality) while leveraging CPPAD for automatic derivative computation.
The solver utilizes CPPAD to automatically obtain derivative information, while CG is employed to generate and store an auto-diff library. This library can be dynamically loaded for efficient evaluation of gradients and Hessians at specified points, significantly reducing computational overhead in subsequent solves. Besides, our implementation supports both parameterized and non-parameterized function definitions, allowing real-time modification of problem parameters (e.g., reference, terminal, or initial states) without the need to regenerate the library. It is noteworthy to point that

\textbf{QP solver.}
The key subroutine in \crisp is to solve the convex trust-region QPs. After careful comparison and experimentation with several QP solvers in \crisp, including \scenario{OSQP}~\cite{osqp}, \scenario{PROXQP}~\cite{proxqp}, \scenario{qpOASE}~\cite{qpOASES}, and \scenario{PIQP}~\cite{schwan2023piqp}, we integrate \scenario{PIQP}, an interior-point method-based QP solver with embedded sparse matrix operations, which strikes an excellent balance of accuracy and real-time performance.
To optimize the construction of the QP subproblems~\eqref{eq:subproblem-smooth}, we perform memory-level operations for extracting, copying, and concatenating large-scale sparse matrices, bypassing the Eigen3 higher-level interfaces. This approach enhances the efficiency of sparse matrix operations that are crucial for complex robot motion planning problems.

To enhance usability, we also implement a Python interface using Pybind11. This interface allows users to adjust problem parameters and solver hyperparameters, solve and re-solve problems, and extract results within a Python environment, providing flexibility for downstream applications.
\input{sections/alg-crisp.tex}
\begin{table}[ht]
    \centering
    \caption{Definition of Hyperparameters in \scenario{CRISP}}
    \begin{tabular}{ cc }
    \toprule[1pt]
    \textbf{Paramaters} & \textbf{Descriptions}\\
    \midrule
    $k_{max}$ & max iteration numbers: 1000 \\
    $\Delta_0$ & initial trust region radius: 1\\
    $\Delta_{max}$ & max trust region radius: 10\\
    $\mu_0$ & initial penalty: 10\\
    $\mu_{max}$ & max penalty: $1e^6$\\
    $\eta_{low}$ & reduction ratio lower bound: 0.25\\
    $\eta_{high}$ & reduction ratio upper bound: 0.75\\
    $\gamma_{shrink}$ & trust region shrink factor: 0.25\\
    $\gamma_{expand}$ & trust region expand factor: 2\\
    $\epsilon_c$ & tolerance of constraint violation: $1e^6$\\
    $\epsilon_p$ & tolerance of trial step norm: $1e^{-3}$\\
    $\epsilon_r$ & tolerance of trust region radius: $1e^{-3}$\\ 
    \bottomrule[1pt]
    \end{tabular}
    \label{tab:hyperparameters}
\end{table}
\begin{remark}[First-order QP Solver]\label{remark:fom_qp}
    There are growing interests in robotics to leverage first-order QP solvers (e.g., building upon ADMM~\cite{stellato20mpc-osqp} and PDHG~\cite{lu2023practical}) for large-scale motion planning. While first-order QP solvers are typically more scalable due to cheap per-iteration cost, recent work~\cite{khazoom2024tailoring} has shown that insufficient solution quality caused by first-order QP solvers may hurt motion planning. For this reason, \crisp chose the interior-point solver PIQP to make sure the inner QPs are solved to sufficient accuracy. However, we do note that \crisp is built in a modular way such that users can switch between different QP solvers. 
\end{remark}


%% file: sections/alg-crisp.tex

\begin{algorithm}[htp]
    \caption{\crisp}
    \label{alg:crisp}
    \SetAlgoLined
    \KwIn{Initial guess $x_0$}
    \KwParameter{$k_{\max}$, $\Delta_0$, $\Delta_{\max}$, $\mu_0$, $\mu_{\max}$, $\eta_{\text{low}}$, $\eta_{\text{high}}$, $\gamma_{\text{shrink}}$, $\gamma_{\text{expand}}$, $\epsilon_c$, $\epsilon_p$, $\epsilon_r$}
    Initialize $k \leftarrow 0$, $\mu \leftarrow \mu_0$, $\Delta_k \leftarrow \Delta_0$\;
    \While{$k < k_{\max}$}{
        Evaluate $f_k$, $\nabla f_k$, $\nabla^2_{xx} f_k$, $c_k$, and $\nabla c_k$ at $x_k$\;
        Construct and solve subproblem~\eqref{eq:subproblem-smooth} to obtain $p_k$\;
        Compute reduction ratio $\rho_k$\;
        // \texttt{\small{Second Order Correction}}\;
        \If{$\textup{ared}_k < 0$}{
            From~\prettyref{eq:second_order_correction}, calculate Second Order Correction Step $\Bar{p}_k$\;
            $p_k \leftarrow \Bar{p}_k$\;
            Re-compute reduction ratio $\rho_k$\;
            \If{$\textup{ared}_k < 0$}
            {Abandon the Step\;
            $\Delta_{k+1} \leftarrow \gamma_{\text{shrink}} \Delta_k$\;
            Go to Line 28\;
            }
        }
            // \texttt{\small{Trust Region Update}}\;
            \uIf{$\rho_k < \eta_{\text{low}}$}{
                $\Delta_{k+1} \leftarrow \gamma_{\text{shrink}} \Delta_k$\;
                $x_{k+1} \leftarrow x_k + p_k$\;
            }
            \uElseIf{$\rho_k > \eta_{\text{high}}$ and $\|p_k\|_{\infty} = \Delta_k$}{
                $\Delta_{k+1} \leftarrow \min\{\gamma_{\text{expand}}\Delta_k, \Delta_{\max}\}$\;
                $x_{k+1} \leftarrow x_k + p_k$\;
            }
            \Else{
                $\Delta_{k+1} \leftarrow \Delta_k$\;
                $x_{k+1} \leftarrow x_k + p_k$\;
            }
        
        // \texttt{\small{Check Convergence of Merit Function}}\;
        \If{$\Delta_{k+1} < \epsilon_r$ \textbf{or} $\|p_k\|_{\infty} < \epsilon_p$}{
        \texttt{\small{Merit Function Converged}}\;
        \texttt{\small{Check Overall Convergence}}\;
                \eIf{max constraint violation $ < \epsilon_c$}{
                    \Return{$x^* \leftarrow x_k$}\;
                    \texttt{\small{Optimization Successful}}\;
                }{
                    \For{$i\in\calE$ with $|c_i(x_k)| \geq \epsilon_c$\\ \textbf{and} $i\in\calI$ with $[c_i(x_k)]^- \geq \epsilon_c$}{
                        $\mu_i \leftarrow \min(10\mu_i, \mu_{\max})$\;
                    }
                    \If{any $\mu_i > \mu_{\max}$}{
                        \Return{Failure}\; 
                        \texttt{\small{Penalty Max Out}}\;
                    }
                }
            
        }
        $k \leftarrow k + 1$\;
    }
    \texttt{\small{Max Iterations Reached\;}}
    \Return{Failure;}
    \end{algorithm}

%% file: sections/result.tex
\section{\crisp: Benchmark}
\label{sec:exp}
In this section, we present a comprehensive benchmark study to evaluate the performance of our proposed solver, \scenario{CRISP}, against several state-of-the-art optimization algorithms. As illustrated in~\prettyref{fig:all_tasks}, we have selected six representative contact-rich planning problems, each involving intricate interactions of normal (support) forces and friction forces. These problems showcase the diverse challenges encountered in practical applications where contact dynamics play a crucial role. All optimization problems are derived using the contact-implicit formulation and the detailed derivations are presented in~\prettyref{app:contact-implicit-formulation-for-experimental-tasks}. The optimization horizon for push T is 50 with time discretization $dt=0.05$ seconds, and all other examples are optimized over 200 steps with $dt=0.02$ seconds. All experiments were conducted on a laptop equipped with a 13th Gen Intel(R) Core(TM) i9-13950HX processor. 

\begin{figure*}[t] 
    \centering
    \includegraphics[width=1.0\linewidth, trim={0cm 0cm 0cm 0cm}, clip]{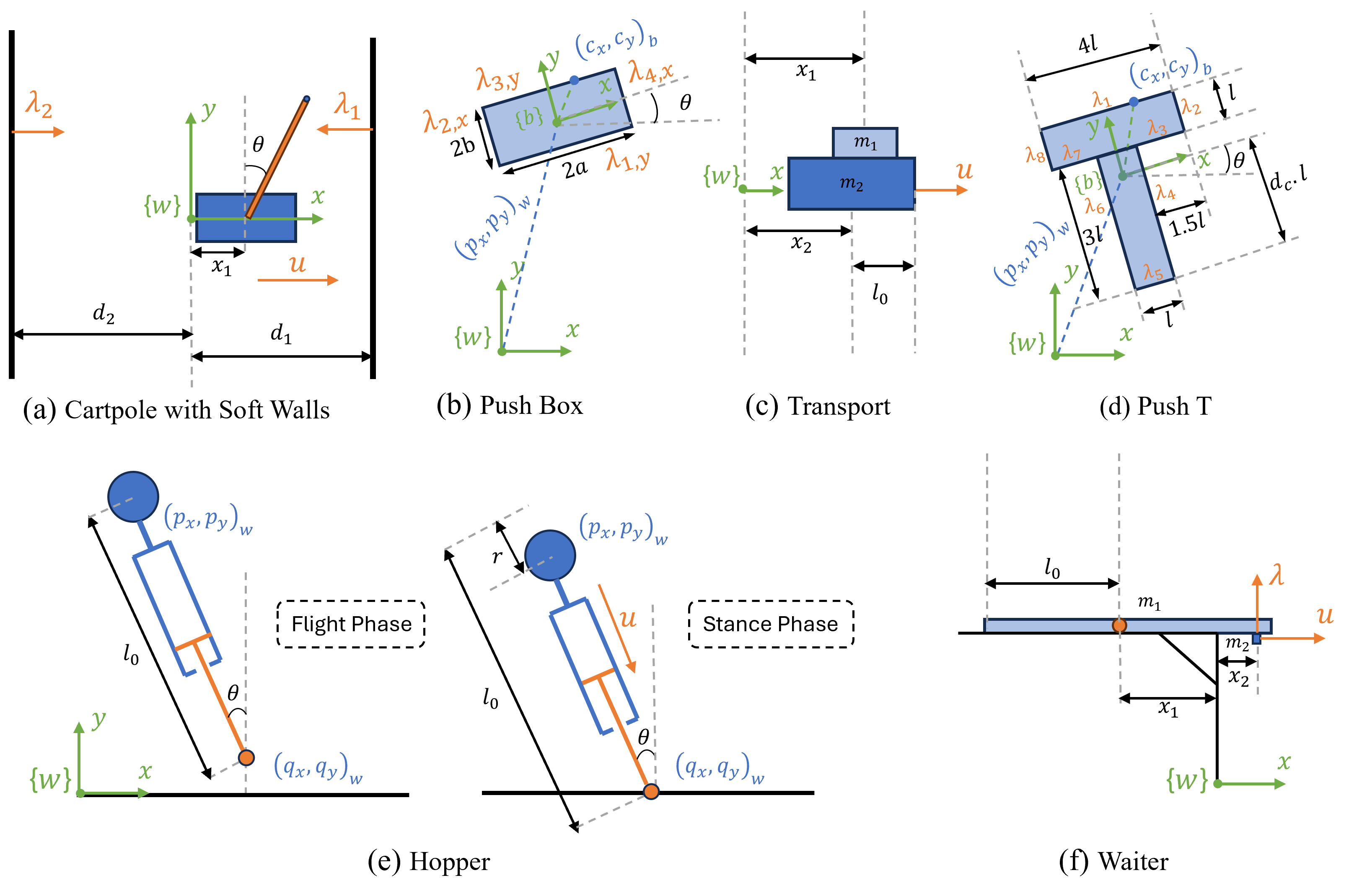}
    \caption{Schematic overview of the contact-implicit motion planning tasks considered in the experiments. Each task poses unique challenges in contact sequencing, force distribution, and modeling of the multi-modal dynamics.}
    \label{fig:all_tasks}
    \vspace{-3mm}
\end{figure*}

\subsection{Benchmark Solvers}
In the first part (\S\ref{sec:comparsion_with_numerical_solvers}), we compare \scenario{CRISP} against three state-of-the-art algorithms.
The first is \snopt \cite{SNOPT}, a sparse nonlinear optimizer based on the SQP method. The second is \ipopt \cite{ipopt}, an interior point optimizer designed for large-scale nonlinear programming. Finally, we include \textsc{ProxNLP} \cite{proxNLP}, a primal-dual augmented Lagrangian algorithm designed by roboticists. 
in subsection \S\ref{sec:additional_exp}, we present additional experiments to further demonstrate \crisp's robustness and superior performance across a broader range of contact-rich motion planning scenarios. Specifically, (\emph{i}) For the push T task, we reformulate the problem as a mixed-integer quadratically constrained problem (MIQCP), which we solve using Gurobi~\cite{gurobi}. (\emph{ii}) We evaluate the LCP solver LCQpow \cite{Hall_2024} on two examples with linear dynamics and linear complementarity constraints. (\emph{iii}) We validate \crisp in real-world Push T with a Franka Panda robot arm. 

\subsection{Benchmark Problems}
For each problem depicted in Fig. \ref{fig:all_tasks}, we provide a brief description of the task along with the initial conditions and initial guesses supplied to the solvers. The detailed formulations for each example, including objective functions, dynamics, contact constraints, and other general constraints, are provided in~\prettyref{app:contact-implicit-formulation-for-experimental-tasks}. Furthermore, we have open-sourced all these problems as examples in \crisp.
\begin{enumerate}
\item Cartpole with soft walls: As shown in Fig. \ref{fig:all_tasks}(a), this problem involves trajectory optimization of a cartpole system constrained between two walls, with the objective of swinging up the pole to an upright stance. The walls are considered ``soft'', providing unidirectional forces analogous to springs. Depending on the initial conditions, the optimal trajectory may require the pole to make or break contact with the walls, thereby leveraging contact forces to accomplish the task. $x_1$ and $\theta$ represent the cartpole states; $u$ is the horizontal active force applied on the cart, while $\lambda_1$ and $\lambda_2$ are the passive contact forces generated when the pole's end makes contact with a wall. These contact forces are proportional to the wall displacement, effectively simulating the soft contact forces. 
Detailed contact-implicit formulation can be found~\prettyref{app:formulation-cartpole}.
We consider two distinct scenarios in the comparison based on the cart's initial position: centered between the walls and in proximity to one wall. For each scenario, we select five different initial conditions, with varying initial pole angles and velocities. For each initial condition, we provide two types of initial guesses: passive trajectories rolled out under zero control input, and these trajectories contaminated with random noise.

\item Push Box: The next benchmark problem is the classic push box problem, as illustrated in Fig.~\ref{fig:all_tasks}(b). In this task, our objective is to manipulate a planar box with dimensions $2a \times 2b$ to a series of specified target configurations on a table with friction. It involves multiple contact modes---four in this case, corresponding to different faces of the box. The solver must reason about both the sequence of contact positions and the appropriate contact forces to achieve the desired motion.
The initial state of the system places the box at the origin, aligned with the positive $x$-axis. We provide an all-zero initial guess to the solvers, and the objective is to push the box to 18 different configurations. Each configuration is located 3 meters from the origin, with the angles distributed evenly in a clockwise direction, completing a full 360-degree circle. 
We refer to~\prettyref{app:formulation-pushbox} for the contact-implicit formulation of different contact modes.

\item Transport: As shown in Fig.~\ref{fig:all_tasks}(c), the goal is to determine the active force \(u\) applied to the cart, under various initial poses and velocities, to transport the payload \(m_1\) to a specified position without it falling off the cart \(m_2\). This requires precise indirect manipulation of \(m_1\) by controlling the friction force between \(m_1\) and \(m_2\) through \(u\). We use an all-zero initial guess for different initial states. The cart is commanded to move from \SI{3} meters to the origin, with different start and goal relative positions (leftmost, middle, rightmost) to the payload (e.g., precisely moving the payload from the leftmost to the rightmost part of the cart while reaching the absolute end position). The task becomes more challenging with different initial and final velocities, requiring alternating dragging motions of the cart to achieve the objective. Its MPCC formulation is provided in Appendix~\ref{app:formulation-transport}.

\item Push T: The modeling techniques for the push T problem are similar to those used in the push box problem. However, it presents its own challenges due to the non-convex shapes of the ``T'' and the increase in contact modes ($\lambda_1-\lambda_8$ as shown in Fig.~\ref{fig:all_tasks}(d)). This means we cannot simply construct complementarity constraints related to contact using four hyperplanes as we do with the push box problem. We now need to establish complementarity conditions that involve contact points lying on certain line segments and exerting corresponding contact forces, which requires special modeling techniques. Details of the problem formulation is demonstrated in Appendix~\ref{app:formulation-pushT}. Together with the slack variables, the total state dimension is 29. The experiment setting is the same with Push Box.
\item Hopper: We model the system as a point mass $m$ atop a massless leg with original length $l_0$, capable of radial contraction and thrust application. As shown in Fig.~\ref{fig:all_tasks}(e), the hopper exhibits two distinct phases: the ``Flight Phase'', where the leg length remains constant and motion approximates free fall, and the ``Stance Phase'', which initiates upon ground contact. During flight, we directly control the leg angular speed ($\dot{\theta} = u_1$). The stance phase, triggered by $p_y - l_0\cos\theta \leq 0$, involves fixed-point rotation around the contact point. As the leg contracts ($r$ increases within $r_0$), it can exert a unidirectional thrust $u = u_2$ along its radial axis. The transition back to flight occurs when $r = 0$.
The task objective is to navigate the hopper's hybrid dynamics without pre-specified contact sequences. The hopper is released from a certain height at the origin, and must optimize its own contact sequence, stance phase contact forces, and flight phase angular velocities to jump and stop at $x = 2\,\textup{m}$. This involves simultaneously managing multiple goals: exploring optimal contact sequences, controlling leg angle during flight, and managing thrust after landing to achieve the specified final position. The initial guess is a passive free-fall motion.
As derived in~\prettyref{app:formulation-hopper}, we employ a contact-implicit formulation to unify the dynamics of both phases within a single framework, with carefully crafted complementarity constraints to govern the transitions between phases.

\item Waiter: The waiter problem, see Fig.~\ref{fig:all_tasks}(f), is also studied in~\cite{RSS-yang-2024} where it was modeled with linear complementarity constraints and optimized in a model-predictive control pipeline. Initially, the plate on the table has only a small overhang. The objective is to indirectly control the friction force with the plate by applying a normal force $\lambda$ and a pulling force $u$ to the plate $m_1$ via the pusher $m_2$. The goal is to gradually extract the plate until its center of mass $x_1$ aligns with the pusher's center $x_2$, stopping at the table's edge with matching velocities. This task requires consideration of the friction between the table and the plate, as well as preventing the plate from tipping over its leftmost contact point with the table (see~\prettyref{app:formulation-waiter}). We provide an all-zero initial guess to test the solvers' ability to optimize an informed contact trajectory.




\end{enumerate}

\subsection{Comparisons with Numerical Solvers }
\label{sec:comparsion_with_numerical_solvers}
While achieving feasible trajectories is a primary goal for local solvers, this alone is insufficient in robotics applications. Dynamically feasible trajectories can often be trivially obtained (e.g., remaining stationary at the origin) but may lack practical value. Therefore, we argue that beyond mere feasibility, the ability to efficiently and robustly optimize informed control sequences for tracking objectives is crucial.

\textbf{Performance metrics.} We evaluate the solvers' performance based on three key metrics. First, we examine the solution success rate, which reflects the overall ability of each solver to converge to a \textit{feasible} and \textit{informed} solution from various settings (different initial/terminal conditions, different initial guesses). Second, we assess the quality of the solution through the objective values achieved and the tracking error. Finally, we measure the solving efficiency by considering the number of iterations required and the total solving time. These metrics serve to answer a fundamental scientific question: 

\emph{Can these solvers robustly discover informed contact sequences and contact force profiles from scratch?} 

This inquiry is particularly relevant, as precise contact patterns are often unknown \emph{a priori} in real-world applications.

\begin{table*}[htbp]
\centering
\resizebox{\textwidth}{!}{%
\renewcommand{\arraystretch}{1.5}
\large
\begin{tabular}{@{}c|ccccc|ccccc|ccccc|ccccc@{}}
\toprule
\multirow{2}{*}{\Large\textbf{Methods}} & \multicolumn{5}{c|}{\Large\textbf{Cartpole with Soft Walls}} & \multicolumn{5}{c|}{\Large\textbf{Push Box}} & \multicolumn{5}{c}{\Large\textbf{Transport}} & \multicolumn{5}{c}{\Large\textbf{Push T}} \\
\cmidrule(lr){2-6} \cmidrule(lr){7-11} \cmidrule(lr){12-16}  \cmidrule(lr){17-21}
 & \stackon{ Rate [\%]}{Success} & \stackon{Error}{Tracking} & \stackon{}{Violation}  & \stackon{}{Iterations} & \stackon{[s]}{Time} & \stackon{ Rate [\%]}{Success} & \stackon{Error}{Tracking} & \stackon{}{Violation} & \stackon{}{Iterations} & \stackon{[s]}{Time} & \stackon{ Rate [\%]}{Success} & \stackon{Error}{Tracking} & \stackon{}{Violation} & \stackon{}{Iterations} & \stackon{[s]}{Time}  & \stackon{ Rate [\%]}{Success} & \stackon{Error}{Tracking} & \stackon{}{Violation} & \stackon{}{Iterations} & \stackon{[s]}{Time}\\
\midrule

\snopt & $63.33$ & $0.08$ & $2.0\times 10^{-3}$ & $4552.50$ & $12.97$ & $0$& $4.25$ & $1.0\times 10^{-11}$ & $3303.00$ & $25.22$ & $\bf{100}$& $1.7 \times 10^{-4}$ & $5.5\times 10^{-9}$ & $5782.30$ & $43.98$ & $33.33$& $0.67$ & $\bf 6.4\times 10^{-8}$ & $2257.50$ & $210.92$ \\
\midrule

\ipopt & $86.67$ & $0.04$ &  $8.9\times 10^{-7}$ & $\bf 288.21$ & $\bf 0.51$ & $94.44$& $0.03$ & $\bf 1.3\times 10^{-11}$ & $\bf 407.69$ & $0.92$ & $87.50$& $0.02$ & $1.6\times 10^{-8}$ & $404.96$ & $1.36$ & $72.22$& $0.26$ & $4.9\times 10^{-7}$ & $1560.20$ & $39.49$\\
\midrule

\scenario{PROXNLP} & $60$ & $0.032$ & $1.4\times 10^{-3}$ & $716.25$ & $140.92$ & $66.67$ & $0.03$ & $3.6\times 10^{-5}$ & $2000.00$ & $1277.21$ & $59.26$ & $0.01$ & $1.1\times 10^{-3}$ & $361.92$ & $1070.01$ & $-$ & $-$ & $-$ & $-$ & $-$\\
\midrule

\crisp (ours) & $\bf{100}$ &  $\bf 0.02$ & $\bf 1.4\times 10^{-7}$ & $415.39$ & $1.56$ & $\bf{100}$& $\bf 0.02$ & $8.3\times 10^{-10}$ & $420.69$ & $\bf{0.74}$ & $\bf{100}$& $\bf 1.1 \times 10^{-4}$ & $\bf 1.1\times 10^{-11}$ & $\bf{286.20}$ & $\bf 0.77$ & $\bf{100}$& $\bf 0.01$ & $8.7\times 10^{-7}$ & $\bf 311.35$ & $\bf 3.54$ \\

\bottomrule
\end{tabular}
}

\caption{Comparison of \crisp with benchmark solvers across different tasks.}
\label{tab:comparison}
\end{table*}

\begin{figure*}[thp]
    \centering
    \includegraphics[width=1.0\linewidth, trim={0cm 0.0cm 0cm 0cm}, clip]{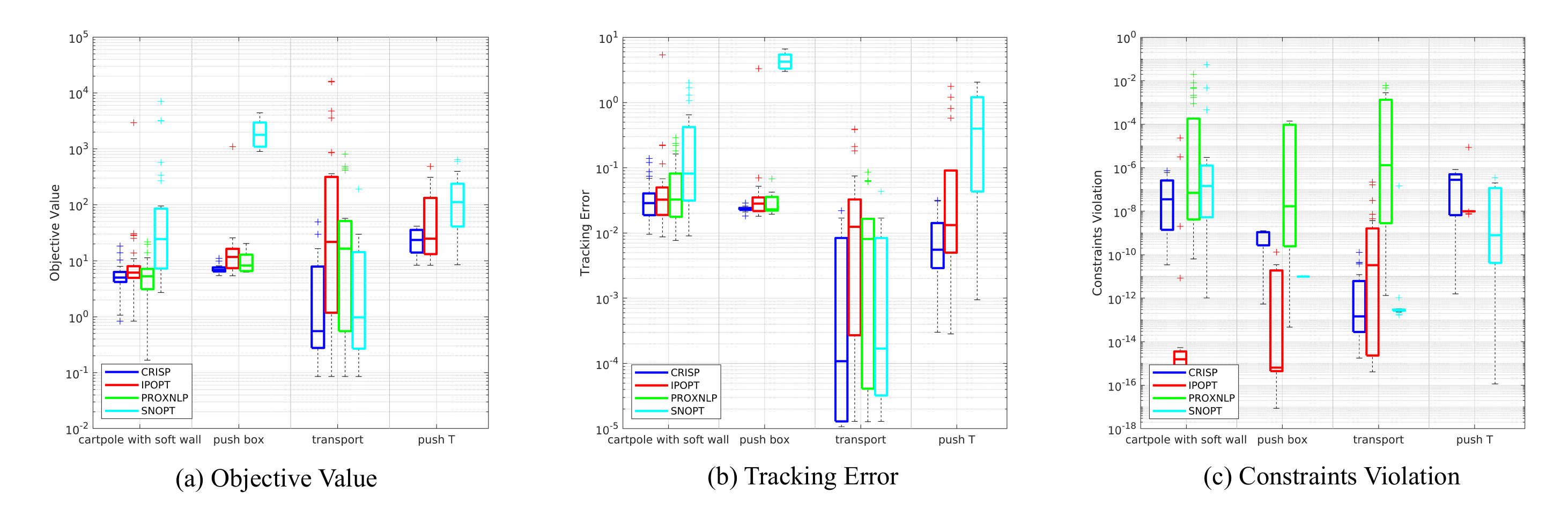}
    \caption{Box plots of benchmark metrics on three tasks: cartpole with soft walls, push box, transport, and push T. The plots complement~\prettyref{tab:comparison} to show the distribution of (a) objective value, (b) tracking error, and (c) constraint violation across multiple initial states and initial guesses. \scenario{PROXNLP} was unable to solve the Push T problem within a tractable amount of time due to its larger problem size, so its data is not applicable in the figure. }
    \label{fig:box_trackingerror}
\end{figure*}

\begin{figure}[thp]
    \centering
    \includegraphics[width=1.0\linewidth, trim={0cm 0cm 0.2cm 0.5cm}, clip]{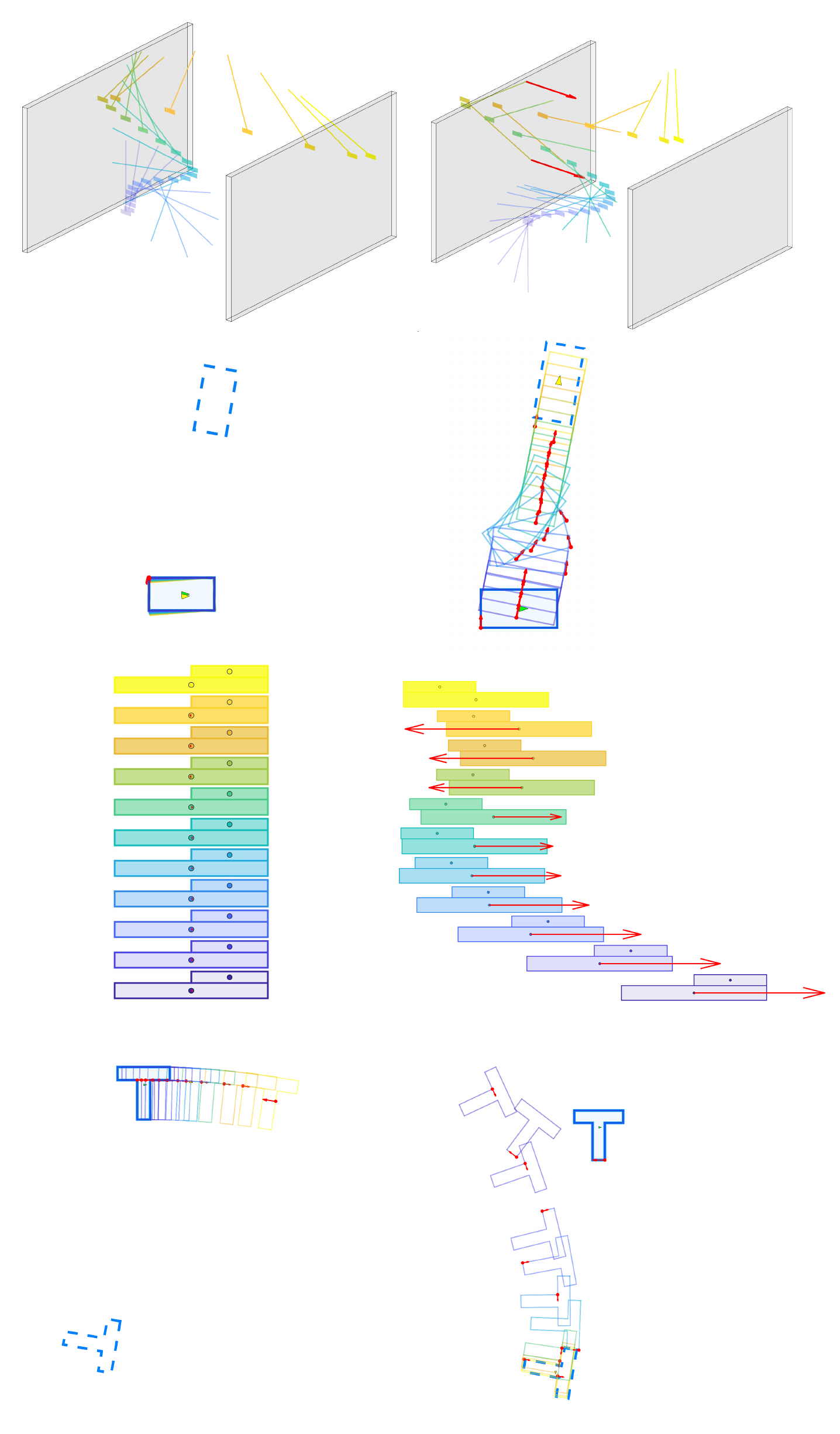}
    \caption{Visualization of some outlier cases in \ipopt. \ipopt results (left column) versus our approach (right column). The color gradient represents the progression of time (from blue to yellow).
    We note that in the transport example (the third row), the initial state has the cargo at the rightmost end of the cart, with both moving left at 4 $\textup{m}/\textup{s}$. The terminal state requires the cart at the origin and the cargo at the leftmost end, both with a leftward velocity of 2 $\textup{m}/\textup{s}$. This scenario necessitates initial braking, followed by back-and-forth motion near the origin. Actually, this challenging case causes all other solvers to fail. While \ipopt produces a feasible but impractical solution (remaining stationary at the origin), our approach successfully optimizes the complex contact sequences.}
    \label{fig:exp_outlier_01}
\end{figure}

\begin{figure}[thp]
    \centering
    \includegraphics[width=1.0\linewidth, trim={1cm 0cm 0cm 0cm}, clip]{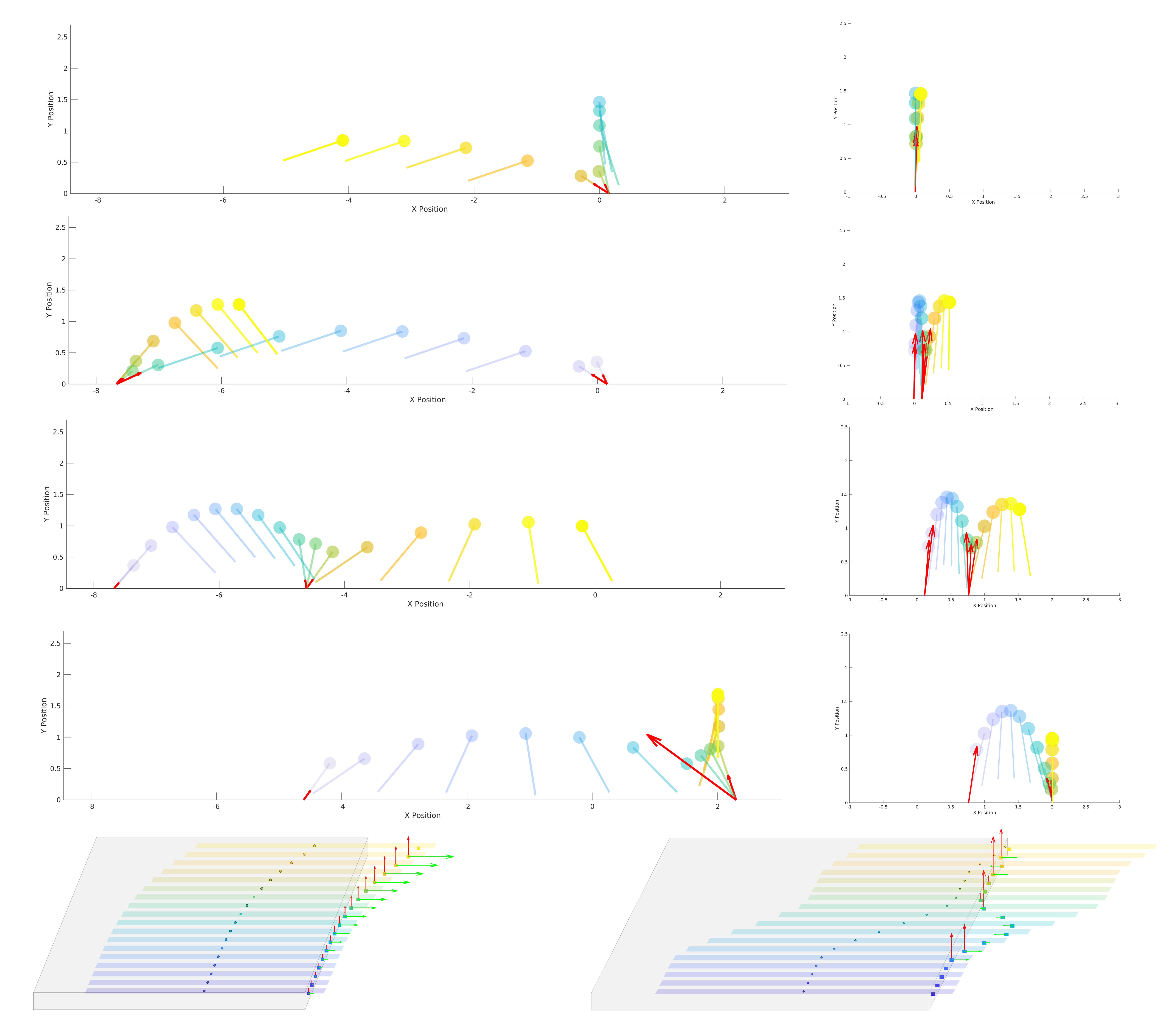}
    \caption{Comparison of optimized trajectories for the hopper and waiter problems. The left and right columns show the results from \ipopt and \crisp respectively. In the hopper problem, we present four key frames (from top to bottom at $1\,\textup{s}$, $2\,\textup{s}$, $3\,\textup{s}$, and $4\,\textup{s}$) illustrating the hopper's trajectories and contact force. The yellow color indicates the current position of the hopper, while blue represents the past trail.}
    \label{fig:exp_hopper_waiter}
\end{figure}

\textbf{Quantitative results.}
For the first four examples, we recorded and summarized the numerical performance of the benchmark solvers under different initial states and initial guesses, as presented in the~\prettyref{tab:comparison} and Fig. \ref{fig:box_trackingerror}. Solutions with constraints violation below $10^{-5}$, translation error norm below 0.1, translation velocity error norm below 0.5, angular error below $\pi/6$, and angular velocity error below $0.1\pi\,\textup{rad}/\textup{s}$ as deemed successful.
The objective function, as previously highlighted, primarily balances tracking error and control efforts. We analyze the performance of four solvers across the first four examples, focusing on median tracking error, average iteration count, and computation time. To better illustrate the performance and robustness of each solver under different initial guesses, we boxplot the distribution of constraint violations, objective values, and tracking errors in Fig. \ref{fig:box_trackingerror}. These plots provide a comprehensive overview of solver behavior across all experiments, including both feasible and infeasible solutions, offering a complete picture of their performance characteristics.

\emph{Our results demonstrate that \crisp consistently produces feasible trajectories with the lowest tracking error and overall objective values across all tasks.} 

The box plot underscores \crisp's remarkable robustness, showing consistently superior performance across varied initial states and guesses, with minimal outliers. While \ipopt generally performs well among the benchmark solvers, it occasionally produces feasible but impractical solutions. We've visualized these outlier cases, contrasting \crisp and \ipopt trajectories in Fig.~\ref{fig:exp_outlier_01}. \snopt, despite outperforming \ipopt in the transport example, struggles with efficiency due to high iteration requirements and exhibits unstable optimization results, notably failing to optimize the informed contact sequence in the push box scenario. These findings underscore the critical importance of solver stability. More visualization of the trajectories can be found in~\prettyref{app:additional-numerical-result-visualize}.

\textbf{Qualitative results. }
For the hopper and waiter examples, which involve more challenging contact reasoning, we focus our comparison on \crisp and \ipopt. The results are visualized in Fig. \ref{fig:exp_hopper_waiter}.
In the hopper task, the objective is to jump and stop at a $2\,\textup{m}$ horizontal distance with zero height, while minimizing control effort. The waiter task requires maneuvering the plate's center of mass to the edge of the table without tipping, ensuring that both the plate and pusher have a final rightward velocity of $2\,\textup{m}/\textup{s}$. Both \ipopt and \crisp successfully optimized feasible trajectories, but \crisp significantly outperformed \ipopt in terms of solution quality for finding informed contact sequences. \crisp achieved remarkably low objective values of 5.45 and 8.68 for the hopper and waiter examples, respectively, with precise tracking performance. In contrast, \ipopt's solutions resulted in much higher objective values of 137.5 and 1855, failing to reach the desired terminal states and incurring substantially larger control efforts, which is evidently shown in Fig. \ref{fig:exp_hopper_waiter}.

\textbf{Videos}. On the project website, we provide animations of optimized trajectories computed by \crisp.

\subsection{Additional Experiments}
\label{sec:additional_exp}
\textbf{Additional Baselines on Push T. }
As an alternative formulation, we reformulated Push T as an MIQCP by associating each facet with a binary variable indicating the contact mode. Using Gurobi to solve this formulation, it failed to find solutions for planning horizons beyond 4 steps (within 10 minutes), making it incomparable with \crisp. The code to reproduce these experiments is open-sourced alongside \crisp.

\textbf{LCP Solvers. }
The Transport and Waiter problems, characterized by linear dynamics and linear complementarity constraints, are naturally suited for a QP + LCP solver. We evaluated LCQpow~\cite{Hall_2024} on these tasks. For the Transport problem, LCQpow achieves an average solve time of 251~s with an average tracking error of 0.4267, and 92.5\% optimized trajectories are feasible. In comparison, as shown in~\prettyref{tab:comparison}, \crisp significantly outperforms LCQpow with an average solve time of 0.77~s, an average tracking error of $1.1 \times 10^{-4}$. For the more challenging Waiter problem, LCQPow fails to generate any feasible solutions. 
The code for reproducing these results is available in our open-source repository.

\textbf{Real-World Push T. }
To further validate the effectiveness of \crisp in delivering high-quality trajectories, we applied \crisp to a real robotic arm system for the push T task. In our validation experiments, we estimated the position of the T block during each perception cycle, using this as the new initial state for \crisp to solve them online. The solution from \crisp was then applied in an MPC feedback loop. In the physical experiments, we discretized the problem into 10 time steps with intervals of 0.05 \textup{s}. Under this setup, the average solution time for \crisp was 80 \textup{ms}, demonstrating robust and efficient performance in pushing the T block to the target position. A visualization of the process for different initial conditions is provided in Fig. \ref{fig:exp_pushT_realworld}.

\begin{figure*}[t] 
    \centering
    \includegraphics[width=1.0\linewidth, trim={0cm 0cm 0cm 0cm}, clip]{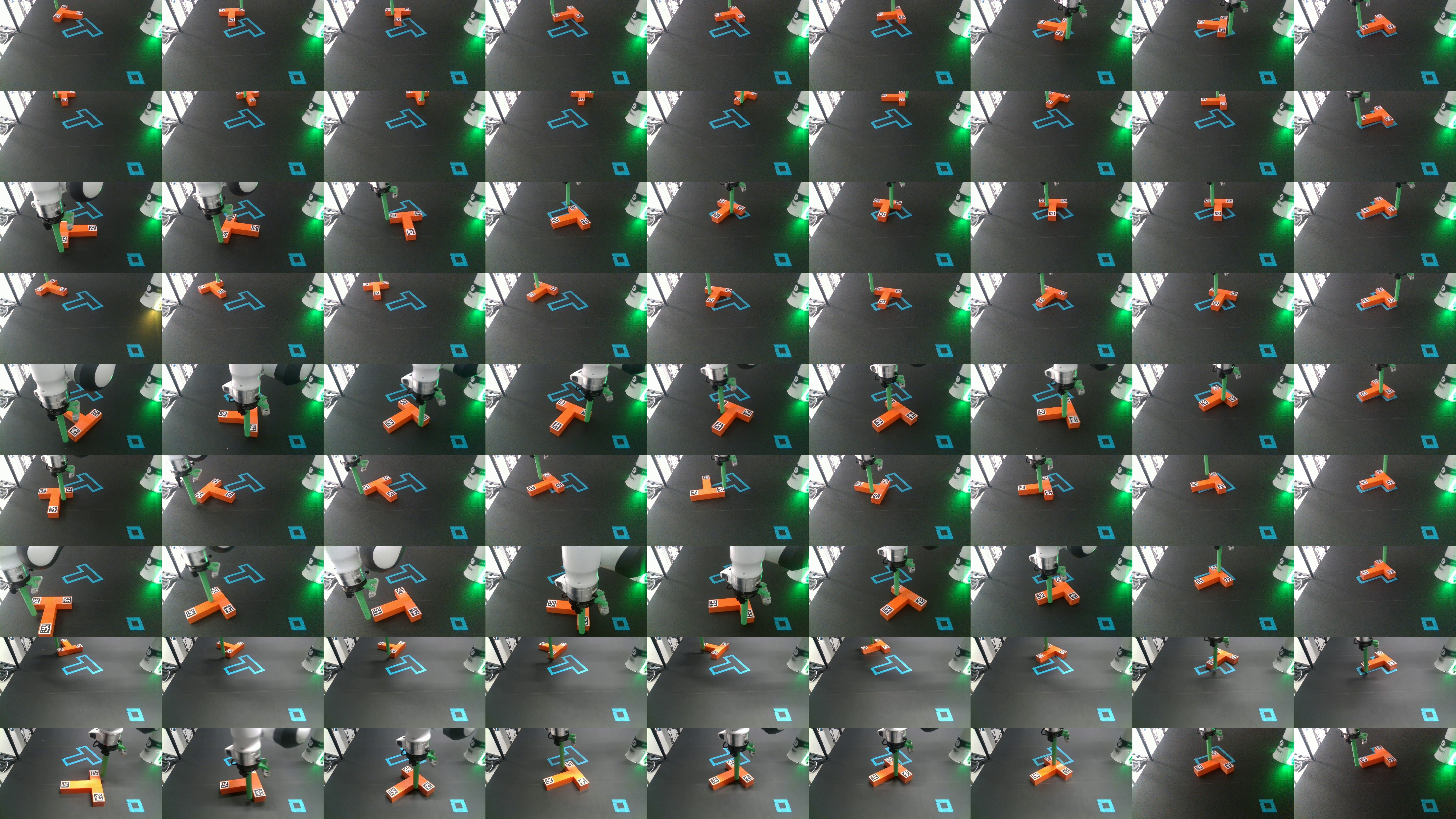}
    \caption{Key frames visualization of the real-world Push T experiments from various initial configurations (each row). Videos of the results can be found on the project website.}
    \label{fig:exp_pushT_realworld}
\end{figure*}

%% file: sections/related-works.tex
\section{Related Works}
\label{sec:relatedworks}
Motion planning through contact presents unique challenges due to the inherently discontinuous nature of contact interactions~\cite{TRO-WENSING-2024,TRO-LELIDEC-2024}. The literature on modeling contact can be broadly divided into smooth and rigid methods. 

\textbf{Smooth contact model. }
The principle of modeling contact in a "smooth" manner involves approximating nonsmooth contact events into smooth and continuous functions relating contact forces to states. This approach often simulates effects similar to springs~\cite{JOB-BLICKHAN-1989}, dampers~\cite{TSMC-MARHEFKA-1999}, or a combination of both~\cite{RAL-NEUNERT-2017,RAL-NEUNERT-2018}. By doing so, it allows contact forces to be expressed as functions of the robot's states and seamlessly integrated into the overall dynamic functions, providing well-defined gradient information.

\textbf{Rigid contact model: hybrid dynamics. } Hybrid systems offer a robust framework for modeling systems that exhibit both continuous and discrete behaviors~\cite{CSM-GOEBEL-2009}. These systems are characterized by their ability to switch between different dynamic regimes, or modes, depending on the contact conditions.
In the locomotion community, the control of switched systems often allows for instantaneous changes in velocity during contact events~\cite{OCS2,HUMANOIDS-FARSHIDIAN-2017,IFAC-FARSHIDIAN20171463}, while continuous dynamics govern the system at other times. This approach requires a predefined gait, which specifies a sequence of potential contact points~\cite{IROS-CHEETAH,ICRA-CHENG-2022,TRO-LOPES-2014}.

\textbf{Rigid contact model: implicit formulations.} 
There are two mathematically equivalent approaches to implicitly encode the discrete nature of hybrid systems for switching between different continuous subsystems. One approach is through Mixed Integer Programming (MIP)~\cite{SIAMReview-vielma-2015}, which introduces binary integer variables to act as switches for encoding contact events~\cite{ICHR-DEITS-2014, IROS-ACEITUNO-2017, RAL-ACEITUNO-2018}. This method is straightforward in its implementation, yet optimizing these discrete variables is challenging for gradient-based NLP solvers and often requires specialized solvers, such as Gurobi~\cite{gurobi}. Furthermore, the number of integer variables can significantly increase with the number of contact modes and the planning horizon, leading to computational intractability~\cite{SHIRAI-YUKI-2024, TAC-MARCUCCI-2021}. On the other hand, the contact force and condition can be encoded through the introduction of complementarity constraints.

On the other hand, contact forces and conditions can be encoded through the introduction of complementarity constraints. Since~\cite{posa2014ijrr-traopt-directmethod-contact}, this approach has recently gained attention because it transforms the problem into a continuous NLP problem without compromising the discrete characteristics of contact. This transformation allows for the use of modern numerical optimization tools to solve the problem effectively~\cite{le2024fast, RSS-yang-2024, aydinoglu2023icra-realtime-multicontact-mpc-admm}. However, the failure of CQs in this context can lead to significant difficulties in solving these problems~\cite{fletcher2000practical, SIOPT-FLETCHER-2006, OMS-Fletcher-2004}. \crisp aims to provide an efficient and robust solution by addressing the challenges associated with solving nonlinear contact problems that include general nonlinear complementarity constraints.

%% file: sections/conclusion.tex
\section{Conclusion}
\label{sec:conclusion}

We presented \crisp, a primal-only numerical solver for contact-implicit motion planning that is based on sequential convex optimization. We started by uncovering the geometric insights underpinning the difficulty of solving MPCCs arising from contact-implicit planning. That motivated us to design a primal-only algorithm where each trust-region subproblem is convex and feasible by construction. For the first time, we proved sufficient conditions on the algorithm's convergence to stationary points of the merit function. With a careful C++ implementation, we benchmarked \crisp against state-of-the-art solvers on six contact-implicit planning problems, demonstrating superior robustness and capability to generate entirely new contact sequences from scratch. We believe \crisp sets a new standard and our open-source benchmarks offer significant value to contact-rich motion planning—one of the most active and challenging areas in robotics research.

\textbf{Limitations and future work. }
First, while \crisp accepts general nonlinear programming problem definitions applicable to various optimal control and motion planning problems, it requires further development to evolve into a comprehensive robotics optimization toolbox comparable to OCS2~\cite{OCS2} and CROCODDYL~\cite{mastalli20crocoddyl}. This development primarily involves integration with advanced dynamics libraries such as Pinocchio~\cite{carpentier2019pinocchio}, which would enable \crisp to handle more complex dynamics and obtain their derivative information efficiently. 
Second, although our C++ implementation is highly efficient, its adaptivity for real-time tasks could be further increased. The acquisition of gradients and Hessian matrices could be parallelized, and matrix operations could achieve significant speedups if implemented on GPUs~\cite{kang2024arxiv-strom}. Additionally, while our subproblems are already cheap as convex QPs, the overall framework's efficiency is still constrained by the speed of the underlying QP solver. As mentioned in Remark~\ref{remark:fom_qp}, we plan to test \crisp with scalable first-order QP solvers. On the theoretical side, a complete picture of the relationship between
local solutions of the original problem and stationary points
of the merit function, as well as how to check if a stationary point of the nonsmooth merit function (that \crisp finds) is locally optimal, warrants further investigation.

\section*{Acknowledgment}

We thank Michael Posa and Zac Manchester for insightful discussions about contact-rich motion planning.



%% file: sections/app-proof.tex

\section{Proof of \prettyref{thm:local_convergence}}
\label{app:proof_main_thm}

\begin{proof}
    As $x_k$ converges to $ x^\star$, we have $\|p_k\| = \|x_{k+1}-x_k\| \rightarrow 0$. By definition, there exists a constant $K$, such that $\norm{p_k}<\Delta_{\mathrm{min}}$ for all $k > K$, indicating that the trust-region constraint becomes inactive in the subproblem. In this proof, we consider only equality constraints, and the proof technique for inequality constraints is similar and will be addressed at the end.
In this case, the merit function~\eqref{eq:merit-func} becomes: 
    \bea
        \phi_1(x;\mu) \triangleq J(x) +  \sum_{i \in \mathcal{E}} \mu_i| c_i(x) |,
    \eea
    where $x \in \Real{n}$, $\calE$ denotes the equality constraints set and $|\mathcal{E}| = m_{\mathcal{E}}$. First, we rewrite the trust-region subproblem~\eqref{eq:subproblem-nonsmooth} to directly optimize over $x_{k+1}$ instead of $p_k$:
    \bea
        x_{k+1} = \displaystyle \argmin_x \{ q_{\mu,k}(x) = J(x) + \sum_{i \in \mathcal{E}}\mu_i \lvert c_i(x_k) + \nabla c_i(x_k)^\top (x - x_k) \rvert \}.
    \eea
    For the $i$th constraint $c_i\in\calE$, we define $w_{k,i}\in\Real{}$ as:
    \begin{equation}
        w_{k,i} \triangleq c_i(x_k) + \nabla c_i(x_k)^\top (x_{k+1} - x_k),
    \end{equation}
    Since the $l_1$ norm $|\cdot|$ is piecewise smooth, we separate the scalar $w_{k,i}$ into three smooth sets:
    \begin{align}
        \{w_{k,i} < 0\}, \{w_{k,i} = 0\}, \{w_{k,i} > 0\}.
    \end{align}
    Extending to the vector of all equality constraints $w_k = \bmat{ccc}w_{k,1}& \dots &w_{k,m_{\calE}}\emat\in\Real{m_\calE}$, we can divide $            \{w_k\}^{\infty}_{k=K}$ to $3^{m_{\calE}}$ smooth sets without overlapping.
    
    Since the number of such sets is finite, there must exist a subsequence $\{w_{n_k}\}^{\infty}_{k=1}$ that belongs to one of these sets, denoting as \( S \). Without loss of generality, we assume that in this set, the first \( m_1 \) components of \( w_{n_k} \) are positive, the \( m_1 + 1 \) to \( m_1 + m_2 \) components are negative, and the remaining components are zero.
    
    For any \( p \in \mathbb{R}^n \), we define the directional derivatives\footnote{We use $\downarrow$ for simplicity, but this is equivalent to $\rightarrow$ because one can easily choose $p$ as $-p$ if $t < 0$.}:
    \begin{align}
        D(\phi_1(x^\star;\mu);p) &\triangleq \lim_{t \downarrow 0} \frac{\phi_1(x^\star + tp;\mu) - \phi_1(x^\star;\mu)}{t}, \\
        D(q_{\mu,n_k}(x_{n_{k}+1});p) &\triangleq \lim_{t \downarrow 0} \frac{q_{\mu,n_k}(x_{n_{k}+1} + tp) - q_{\mu,n_k}(x_{n_{k}+1})}{t}.
    \end{align}
    Since $x_{n_k+1}$ is the global minimizer of the convex function $q_{\mu,n_k}(x)$, we have $D(q_{\mu,n_k}(x_{n_{k}+1});p) \geq 0$. On the other hand,
    \begin{align}
        D(q_{\mu,n_k}(x_{n_{k}+1});p) &= 
         \lim_{t \downarrow 0} \frac{J(x_{n_k+1} + tp) + \sum_{i\in \calE}\mu_i| w_{n_k,i} + t\nabla c_i(x_{n_k})^\top p | - J(x_{n_k+1}) - \sum_{i\in \calE}\mu_i| w_{n_k,i} |}{t} \\
        &= \nabla J(x_{n_k+1})^\top p + \lim_{t \downarrow 0} \frac{\sum_{i\in \calE}\mu_i| w_{n_k,i} + t\nabla c_i(x_{n_k})^\top p | - \sum_{i\in \calE}\mu_i| w_{n_k,i} |}{t}.\\
        & =  \nabla J(x_{n_k+1})^\top p + \sum_{i=1}^{m_1} \mu_i\nabla c_i(x_{n_k})^\top p 
         - \sum_{i=m_1+1}^{m_1+m_2} \mu_i\nabla c_i(x_{n_k})^\top p
        + \sum_{i=m_1+m_2+1}^{m} \mu_i|\nabla c_i(x_{n_k})^\top p|.\\
    \end{align}

    Using the fact that:
    \begin{equation}
        \lim_{t \downarrow 0} \frac{| a + tb | - | a |}{t} = \begin{cases}
            b, & a > 0 \\
            -b, & a < 0 \\
            |b|, & a = 0
        \end{cases}
    \end{equation}
    
    Since $J(x)$ is continuous differentiable and $\nabla c_i$ is Lipschitz, as $k \to \infty$:
    \begin{align}
        \nabla J(x_{n_k+1}) &\to \nabla J(x^\star), \\
        \nabla c_i(x_{n_k}) &\to \nabla c_i(x^\star).
    \end{align}
    
    Thus,
    \begin{align}
        D(q_{\mu,n_k}(x_{n_{k}+1});p) \xrightarrow[k \to \infty]{}
        \nabla J(x^\star)^\top p + \sum_{i=1}^{m_1} \mu_i\nabla c_i(x^\star)^\top p - \sum_{i=m_1+1}^{m_1+m_2} \mu_i\nabla c_i(x^\star)^\top p +
         \sum_{i=m_1+m_2+1}^{m} \mu_i|(\nabla c_i(x^\star)^\top p)| \geq 0.
    \end{align}

    The last inequality comes from the fact that if all the elements in a scaler sequence are nonnegative, then the limitation of the sequence is nonnegative.
    
    Now we turn to calculate $D(\phi_1(x^\star;\mu);p)$. One should be careful: since $S$ is not closed, even the first $m_1$ components of $w_{n_k}$ are always positive, the corresponding components in the converged point $w^\star = \bmat{ccc}c_1(x^\star)&\dots&c_{m_\calE}(x^\star)\emat $ will possibly converge to 0. Similar things happen to the negative part. However, the $m_1 + m_2 + 1 \sim m_\calE$ components of $w^\star$ will always be 0. To fix this, we assume that, for the $1 \sim m_1$ components of $w^\star$, $1 \sim m_{1,+}$ are positive, while $m_{1,+} + 1 \sim m_1$'s components of $w^\star$ become zero. Similarly, $m_1 + 1 \sim m_1 + m_{2,-}$ are negative, while $m_1 + m_{2,-} + 1 \sim m_1 + m_2$'s components of $w^\star$ become zero. Thus,
    
    \begin{align}
        D(\phi_1(x^\star;\mu);p) &= \lim_{t \downarrow 0} \frac{J(x^\star + tp) + \sum_{i \in \calE}\mu_i|c_i(x^\star + tp)| - J(x^\star) - \sum_{i \in \calE}\mu_i|c_i(x^\star)|}{t} \\
        &= \nabla J(x^\star)^\top p + \sum_{i=1}^{m_{1,+}} \mu_i\nabla c_i(x^\star)^\top p + \sum_{i=m_{1,+}+1}^{m_1} \mu_i|\nabla c_i(x^\star)^\top p|\\ 
        &- \sum_{i=m_1+1}^{m_1+m_{2,-}} \mu_i\nabla c_i(x^\star)^\top p+ \sum_{i=m_1+m_{2,-}+1}^{m_1+m_2} \mu_i|\nabla c_i(x^\star)^\top p| 
        + \sum_{i=m_1+m_2+1}^{m_\calE} \mu_i|\nabla c_i(x^\star)^\top p|.\\
        & \geq \nabla J(x^\star)^\top p + \sum_{i=1}^{m_1} \mu_i\nabla c_i(x^\star)^\top p 
         - \sum_{i=m_1+1}^{m_1+m_2} \mu_i\nabla c_i(x^\star)^\top p
        + \sum_{i=m_1+m_2+1}^{m_\calE} \mu_i|\nabla c_i(x^\star)^\top p|.\\
        & \geq 0
    \end{align}
    For inequality constraints we define $w_{k,i}$ as the same. We still separate $\Real{}$ into 
    \bea
    \{w_{k,i} < 0\}, \{w_{k,i} = 0\}, \{w_{k,i} > 0\}
    \eea
    and similarly, we have
    \begin{equation}
        \lim_{t \downarrow 0} \frac{| a + tb |^- - | a |^-}{t} = \begin{cases}
            0, & a > 0 \\
            -b, & a < 0 \\
            |b|^-, & a = 0
        \end{cases}
    \end{equation}
    Other parts of the proof are the same. Consequently $D(\phi_1(x^\star;\mu);p) \geq 0 \quad\forall p$, we conclude the proof.
\end{proof}

%% file: sections/app-formulation.tex
\section{Contact-implicit Formulations}\label{app:contact-implicit-formulation-for-experimental-tasks}
In this section, we provide detailed contact-implicit optimization problem formulations for the five tasks, including tracking objectives, dynamics, contact constraints, and other constraints. We believe this presentation is not only valuable for understanding the principles of contact-implicit modeling, but also represents a necessary step for the model-based community. Interested readers are encouraged to refer to these formulations to test these same problems using their own algorithms. All problem formulations are open-sourced alongside \crisp.
\subsection{Cartpole with Softwalls}\label{app:formulation-cartpole}
\textbf{Dynamics constraints. } Denote the full states of the system as:
$$v = \underbrace{[x, \theta, \dot{x}, \dot{\theta},}_{\text{states}} \underbrace{u, \lambda_1, \lambda_2]}_{\text{control}},$$
and we use subscript $i$ to denote the corresponding variable at time stamp $i$. The dynamics can be written as:
\begin{align}
(m_c + m_p)\ddot{x} + m_p\ell\ddot{\theta}\cos\theta  - (u - \lambda_1 + \lambda_2) &= 0, \label{eq:pushbot:dynamics01}\\
m_p \ell \ddot{\theta} - (\lambda_2 - \lambda_1 - m_p \ddot{x}) \cos\theta - m_p g \sin\theta &= 0.\label{eq:pushbot:dynamics02}
\end{align}
Here, $x$ denotes the cart position, while $m_c$ and $m_p$ represent the mass of the cart and the pole end, respectively. It is important to note that the contact forces $\lambda_1$ and $\lambda_2$ are incorporated into the dynamics equations despite their discrete nature. These forces are automatically triggered and controlled by the complementarity constraints introduced below, which exemplifies the principle and power of contact-implicit formulation. We utilize the Semi-Implicit Euler method \cite{wiki:Semi-implicit_Euler_method} to discretize the dynamics with $dt$. At time stamp $i$:
\begin{align}
    x_{i+1} -x_{i} - \dot{x}_{i+1}dt &= 0,\\
    \theta_{i+1} -\theta_{i} - \dot{\theta}_{i+1}dt &= 0,\\
    \dot{x}_{i+1} -x_{i} - \ddot{x}_{i}dt &= 0,\\
    \dot{\theta}_{i+1} -\theta_{i} - \ddot{\theta}_{i}dt &= 0.
\end{align}
In the above dynamics constraints, $\ddot{x}$ and $\ddot{\theta}$ can be simply derived from~\eqref{eq:pushbot:dynamics01}-(\ref{eq:pushbot:dynamics02}).

\textbf{Contact constraints.}
The complementary constraints from the contact are:
\begin{align}
 0 \leq \lambda_1 &\perp \left(\frac{\lambda_1}{k_1} + d_1 - x - \ell \sin \theta\right)\geq 0, \\
 0 \leq \lambda_2 &\perp \left(\frac{\lambda_2}{k_2} + d_2 + x + \ell \sin \theta\right)\geq 0.
\end{align}
This complementarity constraint is constructed using the relationship between the pole's position and the wall, along with the contact force. Its underlying meaning is that the wall can only provide unidirectional force. When there is no contact, $\lambda$ is zero. Upon contact, the force becomes proportional to the compression distance with coefficients $k$.

\textbf{Initial Constraints.}
The optimization problem is also subject to equality constraints that specify the initial state. 
\begin{align}
    x_0 &= x_{initial}\\
    \theta_0 &= \theta_{initial}\\
    \dot{x}_0 &= \dot{x}_{initial}\\
    \dot{\theta}_0 &= \dot{\theta}_{initial}
\end{align}

\textbf{Cost. }
For the objective function, we opt to minimize both the control effort and the terminal tracking error, each expressed in quadratic form. It is worth noting that intermediate tracking loss is not incorporated into our formulation. This decision stems from the general absence of trivial reference trajectories in contact-involved trajectory optimization tasks. Our goal is to generate a feasible trajectory from scratch that drives the robot as close as possible to the desired terminal state.
\begin{equation}
    f = \frac{1}{2} \sum_{i=1}^{N-1} v_i^T \begin{bmatrix} 0 & 0 \\ 0 & R \end{bmatrix} v_i + \frac{1}{2} v_N^T \begin{bmatrix} Q & 0 \\ 0 & 0 \end{bmatrix} v_N,
\end{equation}
where $Q\in\mathbb{R}^{4\times4}$ and $R\in\mathbb{R}^{3\times3}$ are weighting matrices for the state and control variables. All the default values of the parameters are provided in the example problems of \crisp. 

\subsection{Push Box}\label{app:formulation-pushbox}
We use $p_x$, $p_y$, and $\theta$ to represent the position and orientation of the box in world frame $\{w\}$, while the contact position $(c_x,c_y)$ and contact force $\lambda$ at each facet is defined in the body frame $\{b\}$. We assume that at any given moment, there is only one point of contact between the pusher and the box with only the corresponding normal force applied, and the entire process is quasi-static. The positive direction of the applied forces is defined with respect to the body frame $\{b\}$ of the robot.

\textbf{Dynamics constraints. }
In this work, following \cite{graesdal2024tightconvexrelaxationscontactrich}, we adapt the commonly used ellipsoidal approximation of the limit surface to model the interaction between the contact force applied
by the pusher and the resulting spatial slider velocity. The model captures the principle of the motion of the box while keeping its simplicity.
Denote the full states of the system as: 
$$v = \underbrace{[p_x, p_y, \theta,}_{\text{states}} \underbrace{c_x,c_y,\lambda_{1,y},\lambda_{2,x},\lambda_{3,y}, \lambda_{4,x}]}_{\text{control}}.$$ Then, the push box dynamics can be writen as:
\begin{align}
    \dot{p}_x &= \frac{1}{\mu m g}.\left[(\lambda_{2,x}+\lambda_{4,x})\cos\theta - (\lambda_{1,y} + \lambda_{3,y})\sin\theta\right],\\
\dot{p}_y &= \frac{1}{\mu m g}.\left[(\lambda_{2,x}+\lambda_{4,x})\sin\theta + (\lambda_{1,y} + \lambda_{3,y})\cos\theta\right],\\
\dot{\theta} &= \frac{1}{c r \mu m g}.\left[-c_y(\lambda_{2,x}+\lambda_{4,x}) + c_x(\lambda_{1,y} + \lambda_{3,y})\right],
\end{align}
where $c \in [0,1]$ is the integration constant of the box, and $r$ is the characteristic distance, typically
chosen as the max distance between a contact point and
origin of frame $\{b\}$. Discretize the continuous dynamics with the explicit Euler method, we get:
\begin{align}
    p_{x,k+1} &= p_{x,k} + \dot{p}_{x,k}dt,\\
        p_{y,k+1} &= p_{y,k} + \dot{p}_{y,k}dt,\\    \theta_{k+1} &= \theta_{k} + \dot{\theta}_kdt.
\end{align}

\textbf{Contact constraints. }
First, we ensure that the contact force can only be applied through the contact point $(c_x,c_y)$, and only pointed inward the box, utilizing the contact-implicit formulation:
\begin{align}
 0 \leq \lambda_{1,y} &\perp \left(c_y+b\right)\geq 0, \\
 0 \leq \lambda_{2,x} &\perp \left(c_x+a\right)\geq 0,\\
  0 \leq -\lambda_{3,y} &\perp \left(b-c_y\right)\geq 0,\\ 0 \leq -\lambda_{4,x} &\perp \left(a-c_x\right)\geq 0.
\end{align}
Moreover, to prevent the simultaneous application of forces on adjacent edges at corners, we introduce additional complementarity constraints. These constraints ensure that at any given time, only one force can be active. This is formulated as follows:
\begin{align}
 0 \leq \lambda_{1,y} &\perp \lambda_{2,x}\geq 0, \\
  0 \leq \lambda_{1,y} &\perp -\lambda_{3,y}\geq 0, \\
   0 \leq \lambda_{1,y} &\perp -\lambda_{4,x}\geq 0,\\
    0 \leq \lambda_{2,x} &\perp -\lambda_{3,y}\geq 0,\\
     0 \leq \lambda_{2,x} &\perp -\lambda_{4,x}\geq 0,\\
      0 \leq -\lambda_{3,y} &\perp -\lambda_{4,x}\geq 0.
\end{align}

\textbf{Initial Constraints.}
We add equality constraints to enforce the initial state of the box with $(\Bar{p}_{x,0}, \Bar{p}_{y,0},\Bar{\theta}_{0})$. Note that the initial condition is set to zero in this task.

\textbf{Cost.}
Similar to the setting of cartpole with walls, we adopt a quadratic objective function to penalize the total four contact forces $\lambda$ and the distance to the desired terminal condition $\Bar{v}_N$.
\begin{equation}
    f = \frac{1}{2} \sum_{i=1}^{N-1} v_i^T \begin{bmatrix} 0 & 0 \\ 0 & R \end{bmatrix} v_i + \frac{1}{2} (v_N-\Bar{v}_N)^T \begin{bmatrix} Q & 0 \\ 0 & 0 \end{bmatrix} (v_N-\Bar{v}_N),
\end{equation}

\subsection{Transport}\label{app:formulation-transport}
\textbf{Dynamics constraints:}
Denote all states as $$v = \underbrace{[x_1, x_2, \dot{x}_1, \dot{x}_2, p, q}_{\text{states}} \underbrace{f,u}_{\text{control}}].$$
$f$ is the friction between the two blocks with friction coefficients $\mu_1$; $p$ and $q$ are two slack variables for modeling the direction of the $f$, which is decided by the relative movements between the two blocks. The positive directions of $f$ and $x$ are both horizontally to the right. The dynamics for this problem are straghtforward:
\begin{align}
    m_1\ddot{x}_1 &= f,\\
    m_2\ddot{x}_2 &= u-f.
\end{align}
To avoid cargo $m_1$ from falling off the truck $m_2$, we need:
\beq -l_0 \leq x_1 - x_2 \leq l_0.\eeq

\textbf{Contact constraints.}
The control of the magnitude and direction of friction requires careful handling. $f$ is not an active force; it needs to be indirectly controlled through the manipulation of $u$. This indirect control determines the friction's direction, magnitude, and whether it manifests as static or kinetic friction.
\begin{align}
    \dot{x}_2 - \dot{x}_1 &= p - q,\\
    0\leq p &\perp q \geq 0,\\
    0\leq p &\perp (\mu_1 m_1g -f) \geq 0,\\
    0\leq q &\perp (f+ \mu_1m_1g) \geq 0.
\end{align}
This establishes the relationship between relative velocity and friction force while encoding the transition between static and kinetic friction in a clean way.

\textbf{Initial constraints.}
We add equality constraints to enforce the initial pose and velocity of the cart and the payload.

\textbf{Cost.}
We adopt a quadratic objective function to penalize the active force $u$ and the terminal pos and velocity tracking error.
\subsection{Push T}\label{app:formulation-pushT}
\textbf{Contact-Implicit Formulation:}
Denote the contact point in body frame (located at the COM of T block) as $(c_x,c_y)$, we have 8 different contact mode that exhibit their specific complementarity constraints.
\begin{align}
    -2l\leq &c_x \leq 2l, -d_cl \leq c_y \leq (4-d_c)l,\\
    \lambda_1 \neq 0 &\Rightarrow c_y - (4-d_c)l = 0,\\
    \lambda_2 \neq 0 &\Rightarrow |c_x - 2l| + |c_y-(4-d_c)l| + |c_y-(3-d_c)l| - l = 0,\\
    \lambda_3 \neq 0 &\Rightarrow |c_x - 2l| + |c_x-0.5l| + |c_y-(3-d_c)l| - 1.5l = 0,\\
    \lambda_4 \neq 0 &\Rightarrow |c_x-0.5l| + |c_y-(3-d_c)l| + |c_y + d_cl| - 3l = 0,\\
    \lambda_5 \neq 0 &\Rightarrow |c_y + d_cl| + |c_x + 0.5l| + |c_x - 0.5l| - l = 0,\\
    \lambda_6 \neq 0 &\Rightarrow |c_x + 0.5l| + |c_y-(3-d_c)l| + |c_y + d_cl| - 3l = 0,\\
    \lambda_7 \neq 0 &\Rightarrow |c_x + 2l| + |c_x+0.5l| + |c_y-(3-d_c)l| - 1.5l = 0,\\
    \lambda_8 \neq 0 &\Rightarrow |c_x + 2l| + |c_y-(4-d_c)l| + |c_y-(3-d_c)l| - l = 0.
\end{align}
Designate the positive direction of contact forces as rightward for the $x$
axis and upward for the $y$ axis. And introduce the following slack variables $v$ and $w$ complementing each other to smooth the absolute operation $|\cdot|$ as we did in the transport example:
\begin{align}
    v_1 - w_1 &= c_x -2l,\\
    v_2 - w_2 &= c_y -(4-d_c)l,\\
    v_3 - w_3 &= c_y -(3-d_c)l,\\
    v_4 - w_4 &= c_x -0.5l,\\
    v_5 - w_5 &= c_y + d_cl,\\
    v_6 - w_6 &= c_x + 0.5l,\\
    v_7 - w_7 &= c_x + 2l,\\
    0 \leq v_1 &\perp w_1 \geq 0,\\
    0 \leq v_2 &\perp w_2 \geq 0,\\
    0 \leq v_3 &\perp w_3 \geq 0,\\
    0 \leq v_4 &\perp w_4 \geq 0,\\
    0 \leq v_5 &\perp w_5 \geq 0,\\
    0 \leq v_6 &\perp w_6 \geq 0,\\
    0 \leq v_7 &\perp w_7 \geq 0,\\
    0 \leq v_8 &\perp w_8 \geq 0,
\end{align}
which can be equivalently written as:
\begin{align}
    v_1 + w_1 &= |c_x -2l|,\\
    v_2 + w_2 &= |c_y -(4-d_c)l|,\\
    v_3 + w_3 &= |c_y -(3-d_c)l|,\\
    v_4 + w_4 &= |c_x -0.5l|,\\
    v_5 + w_5 &= |c_y + d_cl|,\\
    v_6 + w_6 &= |c_x + 0.5l|,\\
    v_7 + w_7 &= |c_x + 2l|,\\
\end{align}

Then we have the complementarity constraints:
\begin{align}
    0&\leq -\lambda_1 \perp (4-d_c)l - c_y \geq 0,\\
    0&\leq -\lambda_2  \perp v_1 + w_1 + v_2 + w_2 + v_3  + w_3 - l \geq 0 ,\\
    0&\leq \lambda_3  \perp v_1 + w_1 + v_3 + w_3 + v_4 + w_4 - 1.5l \geq 0,\\
    0&\leq -\lambda_4  \perp v_3 + w_3 + v_4 + w_4 + v_5 + w_5 - 3l \geq 0,\\
    0&\leq \lambda_5  \perp v_4 + w_4 + v_5 + w_5 + v_6 + w_6 - l \geq 0,\\
    0&\leq \lambda_6  \perp v_3 + w_3 + v_5 + w_5 + v_6 + w_6 - 3l \geq 0,\\
    0&\leq \lambda_7  \perp v_3 + w_3 + v_6 + w_6 + v_7 + w_7 - 1.5l \geq 0,\\
    0&\leq \lambda_8  \perp v_2 + w_2 + v_3 + w_3 + v_7 + w_7 - l \geq 0.
\end{align}
Furthermore, if we want to avoid simultaneous application of forces on adjacent edges at corners,
we introduce additional complementarity constraints. These constraints ensure that at any given time, only one force can be active:
\begin{align}
    \lambda_i \perp \lambda_j, \forall i,j = 1\ldots 8, \text{and}\,i\neq j.
\end{align}

With the complementarity constraints derived above, we can write the unified dynamics regardless of the contact point.
\begin{align}
\dot{p}_x &= \frac{1}{\mu m g}.\left[(\lambda_2+\lambda_4+\lambda_6+\lambda_8)\cos\theta - (\lambda_1 + \lambda_3+\lambda_5+\lambda_7)\sin\theta\right],\\
\dot{p}_y &= \frac{1}{\mu m g}.\left[(\lambda_2+\lambda_4+\lambda_6+\lambda_8)\sin\theta + (\lambda_1 + \lambda_3+\lambda_5+\lambda_7)\cos\theta\right],\\
\dot{\theta} &= \frac{1}{c r \mu m g}.\left[-c_y(\lambda_2+\lambda_4+\lambda_6+\lambda_8) + c_x(\lambda_1 + \lambda_3+\lambda_5+\lambda_7)\right],    
\end{align}
where $c \in [0,1]$ is the integration constant of the box, and $r$ is the characteristic distance, typically
chosen as the max contact distance.

\subsection{Hopper}\label{app:formulation-hopper}
The hopper exhibits two-phase dynamics, which we will unify through the complementarity constraints to decide whether it is in the fly or stance.

\textbf{Dynamics constraints. }
We define the full states of the hopping robot as:
$$v = \underbrace{[p_x, p_y, q_x, q_y, \theta, r, \dot{p}_x,\dot{p}_y,}_{\text{states}} \underbrace{u_1,u_2}_{\text{control}}].$$ 
$(p_x, p_y)$ and $(q_x, q_y)$ are the positions of the head and tail of the hopper respectively. $r$ is the leg compression distance. $u_1$ and $u_2$ are the controls for leg angular velocity and thrust, each active only in one phase.
Then, the hopper's dynamics can be written as:
\begin{itemize}
    \item Flight Phase:
    \begin{subequations} \label{eq:2dhopper_flight}
        \begin{align}
        \ddot{p}_x &= 0,\label{eq:2dhopper_flight:01}\\
        \ddot{p}_y &=-g,\label{eq:2dhopper_flight:02}\\
        q_x & = p_x + l_0\sin\theta,\label{eq:2dhopper_flight:03}\\
        q_y &= p_y - l_0\cos\theta,\label{eq:2dhopper_flight:04}\\
        \dot{\theta} &=u_1. \label{eq:2dhopper_flight:05}
    \end{align}
    \end{subequations}
    \item Stance Phase:
    
    In the stance phase, $(q_x, q_y)$ is fixed at the contact point, and we can get the following dynamics equations:
    \begin{subequations}\label{eq:2dhopper_stance}
        \begin{align}
        &m\ddot{p}_x + u_2\sin\theta = 0, \label{eq:2dhopper_stance:01}\\
        &m\ddot{p}_y - u_2\cos\theta + mg = 0,\label{eq:2dhopper_stance:02}\\
        &(l_0-r)\cos\theta - p_y + q_y = 0,\label{eq:2dhopper_stance:03}\\
        &(l_0-r)\sin\theta - q_x + p_x = 0.\label{eq:2dhopper_stance:04}\\
         &(l_0-r)^2-(p_x -q_x)^2 - (p_y - q_y)^2 = 0,\label{eq:2dhopper_stance:05}
        \end{align}
    \end{subequations}
    Also, $q_x$ and $q_y$ are fixed by:
    \begin{subequations}\label{eq:2dhopper_stance_02}
        \begin{align}
        &\dot{q}_x = 0,\label{eq:2dhopper_stance_02:01}\\
        &\dot{q}_y = 0.\label{eq:2dhopper_stance_02:02}
    \end{align}
    \end{subequations}
\end{itemize}
Since during the flight phase, $r$ remains 0, while $u_2$ can only be greater than 0 during the stance phase. We can observe that ~\eqref{eq:2dhopper_flight:01}-\eqref{eq:2dhopper_flight:04} are actually the same as~\eqref{eq:2dhopper_stance:01}-\eqref{eq:2dhopper_stance:04}, except that they involve these phase-specific variables.
To unify the dynamics  into a single optimization problem, we construct the following complementarity constraints:
\begin{subequations}\label{eq:2dhopper_comp01}
    \begin{align}
   0 \leq r \perp q_y \geq 0,\label{eq:2dhopper_comp01:01}\\
   0 \leq u_2 \perp q_y \geq 0,\label{eq:2dhopper_comp01:02}\\
    0 \leq u_1^2 \perp r \geq 0. \label{eq:2dhopper_comp01:03}
\end{align}
\end{subequations}
These constraints ensure that the leg maintains its original length during the flight phase while allowing contraction ($r \geq 0$) when in contact with the ground. These conditions are complementary, enabling us to control the switching of dynamics between phases. Specifically, we transform the dynamics of different phases to:
    \begin{align}
        &m\ddot{p}_x + u_2\sin\theta = 0,\\
        &m\ddot{p}_y - u_2\cos\theta + mg  =0,\\
        &(l_0-r)\cos\theta - p_y + q_y=0,\\
        &(l_0-r)\sin\theta - q_x + p_x=0,\\
        &r.\dot{q}_x = 0,\\
        &r.\dot{q}_y = 0,\\
        &q_y.(\theta-u_1)=0,\\
        &r.\left((l_0-r)^2-(p_x -q_x)^2 - (p_y - q_y)^2\right)=0.
    \end{align}
These unified dynamics, together with the complementary constraints~\eqref{eq:2dhopper_comp01:01}-\eqref{eq:2dhopper_comp01:03}, provide a way to write the hopper problem into one optimization problem in the contact-implicit format.
Discretize the dynamics, we get:
    \begin{align}
        &m\frac{\dot{p}_{x,k+1}-\dot{p}_{x,k}}{dt} + u_{2,k}\sin\theta_k = 0,\\
        &m\frac{\dot{p}_{y,k+1}-\dot{p}_{y,k}}{dt} - u_{2,k}\cos\theta_k + mg  =0,\\
        &(l_0-r_k)\cos\theta_k - p_{y,k} + q_{y,k}=0,\\
        &(l_0-r_k)\sin\theta_k - q_{x,k} + p_{x,k}=0,\\
        &r_k.\frac{q_{x,k+1}-q_{x,k}}{dt} = 0,\\
        &r_k.\frac{q_{y,k+1}-q_{y,k}}{dt} = 0,\\
        &r_k.\left((l_0-r_k)^2-(p_{x,k} -q_{x,k})^2 - (p_{y,k} - q_{y,k})^2\right)=0.
    \end{align}

\textbf{Contact constraints. }
We introduce additional constraints to ensure physical consistency and feasibility:
\begin{equation}
    0 \leq r \leq r_0 \label{eq:2dhopper_leg_length}.
\end{equation}
The leg contraction $r$ within a feasible range. The constant $r_0$ represents the maximum allowable contraction and is chosen such that $r_0 < l_0$, where $l_0$ is the original leg length.

These additional constraints, in conjunction with the complementarity constraint, provide a comprehensive formulation that captures the essential physical characteristics of the hopping robot while maintaining the unified representation of both flight and stance phases.

\textbf{Initial constraints. }
The hopper is released from 1.5\,$\textup{m}$ height with zero initial speed and the leg vertical.

\textbf{Cost. }
The cost is a quadratic loss to force the robot to jump to and stop at 2\,$\textup{m}$ with zero height while penalizing the control effort of the angular velocity $u_1$ and thrust $u_2$.

\subsection{Waiter}\label{app:formulation-waiter}
\textbf{Dynamics constraints. }
Denote the full states in the waiter problem:
$$v = \underbrace{[x_1, x_2, \dot{x}_1, \dot{x}_2, v, w,p,q ,}_{\text{states}} \underbrace{\lambda_N, u,f_p,f_t, N}_{\text{control}}].$$
In this scenario, \((x_1, x_2)\) and \((\dot{x}_1, \dot{x}_2)\) represent the position and velocity of the plate and pusher, respectively. The variables \(z, w, p,\) and \(q\) are slack variables used to model the friction between the plate and table, as well as between the pusher and plate, similar to the transport example. For control, we have direct control over the normal force applied by the pusher, denoted as \(\lambda_N\), and the horizontal thrust, \(u\). Additionally, there is indirect control over the frictional forces between the plate and table, \(f_t\), and between the pusher and plate, \(f_p\). The support force exerted by the table is \(N\). The system dynamics are described by the following equations:

\begin{subequations}\label{eq:waiter}
    \begin{align}
    & m_2 \ddot{x}_2 = u - f_p, \\
    & m_1 \ddot{x}_1 = f_p - f_t, \\
    & N + \lambda_N = m_1 g, \label{eq:waiter:staticforcebalance} \\
    & \lambda_N (x_2 - x_1 + l_0) \leq m_1 g l_0, \label{eq:waiter:torquebalance} \\
    & N \geq 0, \ \lambda_N \geq 0, \\
    & x_2 \geq 0, \ x_2 - x_1 \leq l_0. \label{eq:waiter:pusherposlimit}
    \end{align}
\end{subequations}

Here, ~\eqref{eq:waiter:staticforcebalance} ensures vertical force balance. ~\eqref{eq:waiter:torquebalance} prevents tilting around the leftmost contact point between the table and plate. Finally, ~\eqref{eq:waiter:pusherposlimit} restricts the pusher's position to remain clear of the table while staying on the overhanging part of the plate.

\textbf{Contact constraints. }
The following complementarity constraints control the magnitude and direction of the friction $f_t$ and $f_p$ 
\begin{align}
    \dot{x}_1 &= z - w,\\
    \dot{x}_2 - \dot{x}_1 &= p - q,\\
    0\leq z &\perp w \geq 0,\\    
    0\leq p &\perp q \geq 0,\\
    0\leq z &\perp (N\mu_1 - f_t) \geq 0,\\
    0\leq w &\perp (N\mu_1 + f_t) \geq 0,\\
    0\leq p &\perp (\mu_2\lambda_N - f_p) \geq 0,\\
    0\leq q &\perp (\mu_2\lambda_N +f_p) \geq 0.
\end{align}
These equations ensure the friction cone constraints, which automatically manage the magnitude, nature (static or kinetic), and direction of the frictions.

\textbf{Initial constraints.}
In the waiter problem, the initial state is defined with the pusher positioned next to the table, and the 14-meter-long plate has a 1-meter overhang.

\textbf{Cost. }
The objective is to achieve a terminal position and velocity of the plate and pusher, such that the plate is pulled out until its COM aligns with the pusher next to the table, while also minimizing control efforts.

%% file: sections/app-visualization.tex
\newpage
\section{Visualization of Numerical Results}\label{app:additional-numerical-result-visualize}
\begin{figure*}[h] 
    \centering
    \includegraphics[width=0.8\linewidth, trim={0cm 0cm 0cm 0cm}, clip]{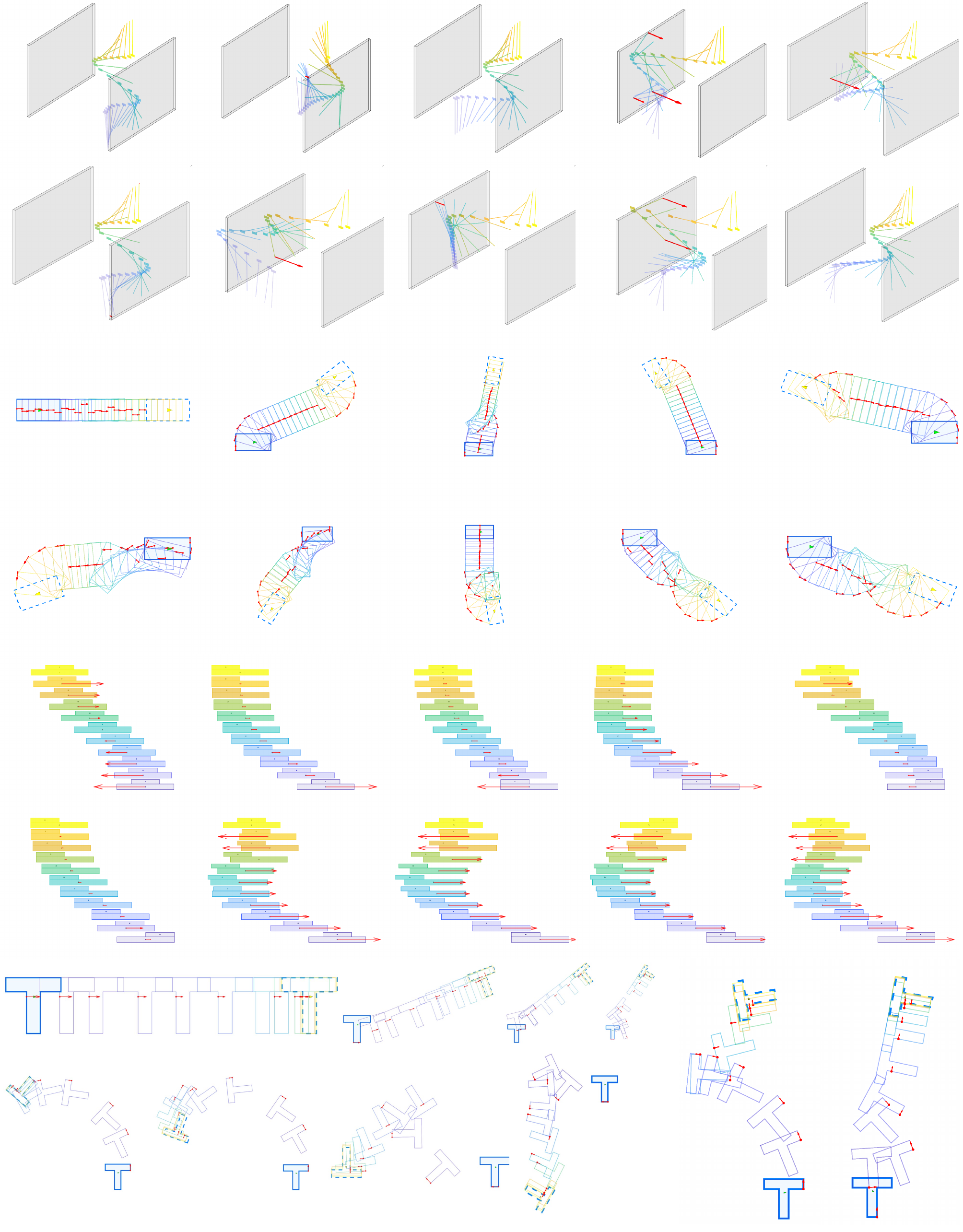}
    \caption{Visualization of more trajectories solved by \crisp in cartpole with soft walls (row 1-2), push box (row 3-4), payload transport (row 5-6), push T (row 7-8) under different initial conditions.}
    \label{fig:exp}
\end{figure*}